\documentclass[12pt]{amsart}
\usepackage{amsmath,amsthm,amssymb,color,enumerate}
\usepackage{graphicx}
\usepackage[normalem]{ulem}
\usepackage{amsrefs}
\usepackage{hyperref}
\usepackage{float}
\usepackage{caption}
\usepackage{subcaption}
\usepackage{amscd}
\usepackage{longtable}
\usepackage{graphicx}
\usepackage{comment}
\usepackage{tikz-cd}
\hypersetup{colorlinks=true,linkcolor=blue,citecolor=blue}

\usepackage{mathtools}
\usetikzlibrary{cd}
\usetikzlibrary{decorations.pathmorphing}

\usepackage{amscd}

\newcommand{\Syz}{\mathrm{Syz}}
\newcommand{\init}{\mathrm{init}}

\newcommand{\tr}{\mathrm{Tr}}
\newcommand{\codim}{\mathrm{codim}}
\newcommand{\id}{\mathrm{id}}
\newcommand{\spec}{\mathrm{Spec}}
\newcommand{\cone}{\mathrm{Cone}}
\newcommand{\conv}{\mathrm{Conv}}
\newcommand{\sing}{\mathrm{Sing}}
\newcommand{\proj}{\mathrm{Proj}}
\newcommand{\dem}{\mathrm{Dem}}
\newcommand{\cl}{\mathrm{Cl}}
\newcommand{\br}{\mathrm{Br}}

\newtheorem{theorem}{Theorem}[section]
\newtheorem{proposition}[theorem]{Proposition}
\newtheorem{lemma}[theorem]{Lemma}
\newtheorem{corollary}[theorem]{Corollary}
\newtheorem{definition}[theorem]{Definition}

\newtheorem{example}[theorem]{Example}

\renewcommand{\caption}[1]{\singlespacing\hangcaption{#1}\normalspacing}

\title {On the Gorensteinization of Schubert Varieties via Boundary Divisors}
\author {Sergio Da Silva}
\address{Sergio Da Silva, Cornell University, Ithaca NY}
\email{smd322@cornell.edu, sergio.dasilva@umanitoba.ca}

\thanks{Research supported in part by an NSERC PGS-D3 scholarship and a PIMS postdoctoral fellowship}

\begin{document}

\begin{abstract}
We will describe a one-step ``Gorensteinization'' process for a Schubert variety by blowing-up along its boundary divisor. The local question involves Kazhdan-Lusztig varieties which can be degenerated to affine toric schemes defined using the Stanley-Reisner ideal of a subword complex. The blow-up along the boundary in this toric case is in fact Gorenstein. We show that there exists a degeneration of the blow-up of the Kazhdan-Lusztig variety to this Gorenstein scheme, allowing us to extend this result to Schubert varieties in general. The potential use of this one-step Gorensteinization to describe the non-Gorenstein locus of Schubert varieties is discussed, as well as the relationship between Gorensteinizations and the convergence of the Nash blow-up process in the toric case.
\end{abstract}\bigskip

\maketitle


Gorenstein varieties provide useful representatives for birational equivalence classes as their canonical (and anticanonical) bundles are invertible. In liaison theory for example,  a union of subschemes is often assumed to be Gorenstein so that properties of these linked subschemes are easier to formulate (see \cite{MN-Linkage}). It is therefore useful to find methods to achieve this marked improvement, a so called \textit{Gorensteinization}. Such a Gorensteinization can be difficult to describe, even for toric varieties in general, except in the simplest cases. The fact that Schubert varieties can be locally degenerated to these simple toric cases allows for the existence of a simple Gorensteinization by blowing-up an explicit Weil divisor.

Schubert varieties are a well-studied class of schemes that have useful descriptions which reduce many otherwise difficult operations to simple combinatorics. Let $G$ be a simple complex Lie group, and fix a maximal torus $T$ as well as a Borel subgroup $B$ containing $T$. By the Bruhat decomposition, $B$ acts on $G/B$ with finitely many orbits indexed by $W$, where $W=N_G(T)/T$ is the Weyl group of $G$. The closures of these orbits $X^w:=\overline{BwB/B}$ are called Schubert varieties.

We can define the boundary of $X^w$ as $$\partial(X^w) = \displaystyle\bigcup_{v\lessdot w}X^v,$$
using $v$ which are covered by $w$ in strong Bruhat order. There is a Frobenius splitting on $X^w$ for which $\partial(X^w)$ is compatibly split (see \cite{Brion-Kumar}). We can localize our question and reduce to blowing-up a Kazhdan-Lusztig variety $X^w\cap X_v^{\circ}$ along its boundary. Here $\partial(X^w)\cap X_v^{\circ}$ is an anticanonical divisor, and it is Cartier if and only if blowing it up is an isomorphism. This makes it useful for detecting the Gorenstein property.

In \cite{Knut}, it was shown that there exists a degeneration from a Kazhdan-Lusztig variety to a Stanley-Reisner scheme associated to a subword complex. Subword complexes were introduced in \cite{KM} and provide a geometric way of viewing subwords of a word written in elements of Coxeter group. Here in the toric case, we can identify the boundary divisor as the topological boundary of a convex region in the character lattice, and it is especially easy to visualize the blow-up along this divisor.

\begin{theorem}\label{SR-gorenstein}
Let $Y = \spec(k[x_1,...,x_n]/I)$ where $I$ is a square-free monomial ideal. If $Y$ is homeomorphic to a ball or sphere, then the blow-up along $\partial Y$  is Gorenstein.
\end{theorem}

In fact, the blow-up of $Y$ from Theorem \ref{SR-gorenstein} has an exceptional divisor whose reduction is the natural candidate for the boundary divisor of $\widetilde{Y}$. We make use of the Frobenius splitting on $Y$ and $\widetilde{Y}$ to determine what an anticanonical divisor for a reducible scheme should be.  In fact $\partial \widetilde{Y}$ will be anticanonical under this definition. We check that it is Cartier by observing that the blow-up of $\widetilde{Y}$ along its boundary is an isomorphism. Hence $\widetilde{Y}$ is Gorenstein. Since being Gorenstein is open in flat families (even with this generalized version -- see Section \ref{semicontinuous}), the blow-up of the Kazhdan-Lusztig variety is also Gorenstein, proving our main theorem.

\begin{theorem}
The blow-up of a Schubert variety $X^w$ along $\partial X^w$ is Gorenstein.
\end{theorem}

This last implication requires a degeneration of the total transform of the Kazhdan-Lusztig variety to the total transform of the degeneration. Blow-ups  however do not in general commute with degenerations. We show that we can choose a Gr\"{o}bner degeneration of $\widetilde{X^w\cap X_v^{\circ}}$ that commutes with blowing-up along the boundary divisor. This result is a more general tool that defines a monomial weighting on a blow-up algebra that degenerates to the blow-up algebra of a given degeneration.

\begin{theorem}
Let $>_{\lambda}$ be a term order on $S=k[x_1,...,x_k]$ defined by an integral weight function $\lambda: \mathbb{Z}^k\rightarrow \mathbb{Z}$. Let $I=\langle g_1,...,g_m\rangle\subset J=\langle g_{1},...,g_{m+n}\rangle$ be two ideals in $S$ each generated by a Gr\"obner basis (with respect to $\lambda$). Then there exists an integral weight function $\tilde{\lambda}:\mathbb{Z}^{n+m+k}\rightarrow \mathbb{Z}$ defining a term order $>_{\tilde{\lambda}}$ on the blow-up algebra of $S/I$ along $J/I$ which degenerates to the blow-up algebra of $ S/\init_{>_{\lambda}}(I)$ along $\init_{>_{\lambda}}(J)/\init_{>_{\lambda}}(I)$.
\end{theorem}

In summary, we show that the blow-up of a Schubert variety along its boundary is Gorenstein. This Gorensteinization will prove even more useful if we can explicitly describe the total transform $\widetilde{X^w}$. On a local level, equations for the blow-up algebra requires an understanding of syzygies. The boundary divisor is the union of Schubert varieties, which are each defined using determinantal polynomial equations. The union is therefore defined by the intersection of these conditions. This local question  reduces to understanding what the syzygies are between products of determinants. While the syzygies between the $k\times k$ minors of a given matrix are well-understood (and defined in terms of determinants -- see \cite{Ma}), it is not known whether the syzygies between products of these minors are also determinantal in nature. It is more likely that finding  $\widetilde{X^w}$ through other means is easier and would shed light on this syzygy problem. It turns out that Bott-Samelson varieties provide a possible remedy.

Bott-Samelson varieties are commonly used as a desingularization for Schubert varieties. Bott-Samelson maps are described in Section \ref{BSR}. They have combinatorial properties whose structure is well understood. Using the universal property of blow-up maps, we show that there exists a surjective map from a generalized Bott-Samelson variety $BS^Q$ to $\widetilde{X^w}$. A topic for future research is to determine to what extent $\widetilde{X^w}$ is isomorphic to some piece of $BS^Q$. The fact that $\widetilde{X^w}$ is weakly normal simplifies the problem considerably.

\begin{proposition}
A Kazhdan-Lusztig variety, its degeneration to a Stanley-Reisner scheme, and the blow-ups along their respective boundaries are weakly normal.
\end{proposition}

This extension of Frobenius splittings to the total transform was studied in \cite{LMP} when considering blow-ups of smooth varieties along smooth centers. Whether the Frobenius splitting extends in our case is not known, so the result in the proposition uses a different method. The weakly normal property however allows for a critical simplification (using a Zariski's main theorem type of argument).

\begin{theorem}
There exists a surjective birational $B$-equivariant morphism $\psi_Q:BS^Q\rightarrow \widetilde{X^w}$ which is an isomorphism iff it is a bijection on $T$-fixed points.
\end{theorem}

One might hope that we can at least understand the exceptional components in the blow-up of a Stanley-Reisner scheme (associated to a subword complex) along its boundary. It turns out that we can provide a combinatorial description for the new facets corresponding to exceptional components, and as a result we can deduce a useful criteria for determining when a Kazhdan-Lusztig variety is Gorenstein. This is the content of Section \ref{SR_Exceptional}.

\begin{theorem}\label{max_face}
Let $\sigma_{Q\setminus P_1}$ be a facet of the subword complex $\Delta(Q,\pi)$. Then $\tau_{Q\setminus P_2}$ is a maximal face of $\sigma_{Q\setminus P_1}\cap\partial\Delta(Q,\pi)$ if and only if $P_2$ is minimal among subwords $P\supset P_1$ satisfying $Dem(P)\gtrdot \pi$. Furthermore,  if $P_1$ has codimension $>1$ in $P_2$ then $\tau_{Q\setminus P_2}$ corresponds to a new facet of $\partial\widetilde{\Delta(Q,\pi)}$.
\end{theorem}

\begin{corollary}\label{KL_Gorenstein}
Consider the Kazhdan-Lusztig variety $X_w\cap X^v_{\circ}$ and let $Q$ be a reduced word for $v$ written as an ordered list of simple reflections.  Then $X_w\cap X^v_{\circ}$ is not Gorenstein if and only if there is a reduced subword $P_w\subset Q$ for $w$ and a codimension $>1$ subword $P\supset P_w$ that minimally satisfies $\dem(P)\gtrdot w$.
\end{corollary}

Finally, let us justify why it is worth studying Gorensteinizations in general. We offer two applications which are topics for future study.\\

\textbf{Then Gorenstein Locus of $X^w$:} In the $GL_n$ case, we can describe which $X^w$ are Gorenstein using pattern interval avoidance. The Gorenstein locus is only conjecturally described in such a way (see \cite{Woo-Yong}). Pattern avoidance conditions for $X^w$ other than type $A$ are not known.

The singular locus of $X^w$ on the other hand is well-understood. One approach is using the quasi-resolutions of Cortez in \cite{Cortez2} and \cite{Cortez}. Here a family of these quasi-resolutions $\pi_i:Y_i\rightarrow X^w$ are chosen, where each $Y_i$ is similar to the Bott-Samelson construction. The singular locus of $X^w$ can then be described as

$$ \bigcap_i (\pi_i(\sing(Y_i))\cup \br(\pi_i))  $$
where $\br(\pi_i)$ is the branch locus of $\pi_i$ and $\sing(Y_i)$ is the singular locus of $Y_i$.

Using this as motivation, we observe that in our case we have a Gorensteinization $\pi:\widetilde{X^w}\rightarrow X^w$ for which we hope that the non-Gorenstein locus can be described in terms of the non-Gorenstein locus of $\widetilde{X^w}$ and the branch locus of $\pi$. Of course the former is empty, so understanding the branch locus of $\pi$ is key (and difficult at this point).

We thank Alexander Woo and Alexander Yong for observing this connection.\\

\textbf{Resolving singularities via Nash Blow-ups:} Given an $m$-dimensional quasiprojective variety $X\subset \mathbb{P}^n$, we can define a Nash blow-up as the closure of the graph of the Gauss map,

$$X \rightarrow Gr(m+1,n+1) $$
taking a smooth point to its tangent space. The question is whether repeating this process terminates after finitely many steps (in which case we have resolved the singularities of $X$). The problem remains open for toric varieties, although work has been done on this problem with all experimental results converging in finite time (see \cite{Proudfoot}).

In the Gorensteinization process of $X^w$, we degenerate to the toric case and consider the method of blowing-up boundary divisors until we have an isomorphism. While we don't know whether this process would terminate for toric varieties in general, we develop ideas towards this result in the sections that follow.

The interesting connection is that this Gorensteinization process is linked to the Nash blow-up problem. We thank John Moody for the observation that the Gorensteinization process converging to a smooth variety implies that the Nash Blow-up process also converges. The convergence of this Gorensteinization process corresponds to a torus equivariant sheaf of finite type that also contains a sheaf which determines the convergence of the Nash blow-up process. These are both subsheaves of the Grauert-Riemenschneider sheaf.

This exciting link between Gorensteinizations for toric varieties and resolving their singularities via Nash blowups is certainly worth pursuing in future research.\newline

\textbf{Acknowledgments:} I would like to give special thanks to my advisor Allen Knutson for his direction and guidance all of these years. I have learned a great deal from him and am thankful for the many projects we have worked on. I would also like to thank my thesis committee members Yuri Berest, Michael Stillman and Ed Swartz whose support during this process has been invaluable. Finally, thanks to  Michel Brion, Karl Schwede, Alexander Woo and Alexander Yong for their advice on a variety of topics related to this research.

\section{Preliminaries}

For us, a scheme will be separated of finite type over an algebraically closed field $k$ of characteristic 0, unless specified otherwise. A variety will be an integral scheme.

Let $G$ be a simple complex Lie group, and fix a maximal torus $T$ as well as a Borel subgroup $B$ containing $T$. By the Bruhat decomposition, $B$ acts on $G/B$ with finitely many orbits indexed by $W$, where $W=N_G(T)/T$ is the Weyl group of $G$, and

$$G/B = \bigsqcup_{w\in W} BwB/B \text{  .}$$

For each $w\in W$, let $X^w_{\circ} = BwB/B$, called a Schubert cell. Then its Zariski closure is the Schubert variety $X^w$ (note the difference in the notation $X_w$ used by some other authors). The left $T$ action on $X^w$ has finitely many $T$-fixed points $e_v :=vB/B$ for $v\leq w$. Every point in $X^w$ is contained in the $B$-orbit of some $e_v$.

We can define the \textbf{boundary} of $X^w$ using $v$ which cover $w$ in strong Bruhat order as

$$\partial X^w = \displaystyle\bigcup_{v\lessdot w}X^v \text{  .}$$

\begin{example}\normalfont In the $G = GL_4(\mathbb{C})$ case, it is not hard to see that the boundary of $X^{4231}$ has four components given by

$$\partial X^{4231} = X^{4213} \cup X^{4132} \cup X^{3241} \cup X^{2431}$$

Recall that the dimension is computed by counting the number of inversions in the permutation, so here $$\dim X^{4231} =5 \text{ and }\dim \partial X^{4231} =4.$$\qed
\end{example}

Finally, we state a well-known result for reference.

\begin{lemma}\label{schubert-normal}
The Schubert variety $X^w$ is Cohen-Macaulay and normal.
\end{lemma}

\subsection{Kazhdan-Lusztig varieties}\label{KL}

To get local equations for $X^w$, we resort to computing Kazhdan-Lusztig ideals (see \cite[\S 3]{Woo-Yong}) using local coordinates for $G/B$. There is an isomorphism between a neighborhood of any point in $G/B$ with a neighborhood of a $T$-fixed point (using the $B$-action on $G/B$). Let $X_w^{\circ}:= B_{-}wB/B$ denote the opposite Schubert cell. To find local equations for $X^w$ at $e_v$, it is enough to study $X^w\cap vX^{\circ}_{\id}$, where $vX^{\circ}_{\id}$ is an affine neighborhood of $e_v$.

\begin{lemma}[Kazhdan-Lusztig Lemma]
$X^w\cap vX^{\circ}_{\id}\cong (X^w\cap X^{\circ}_v)\times \mathbb{A}^{l(v)}.$
\end{lemma}
\begin{proof}
See 3.2 in \cite{Woo-Yong}.
\end{proof}

The $X^w\cap X^{\circ}_v$ are called Kazhdan-Lusztig varieties. We justify the notation $X^w$ instead of $X_w$ because of the fact that $ X^w\cap X^{\circ}_v\neq \emptyset$ iff $w\geq v$ (and conventionally posets are drawn with larger elements on the top).

Since $X^w$ is covered by open sets of the form $X^w\cap vX^{\circ}_{\id}$, it will be enough to study $X^w\cap X^{\circ}_v$ to deduce information about $X^w$. We will denote the defining ideal for the Kazhdan-Lusztig variety $X^w\cap X_{v}^{\circ} $ by $I_{w,v}$. A simple topological computation is needed to restrict $\partial X^w$ to the local case:

\begin{lemma}
$\partial((X^w\cap X_v^{\circ})\times \mathbb{A}^{l(v)}) = (\partial X^w\cap X_v^{\circ})\times \mathbb{A}^{l(v)}.$
\end{lemma}

We will later see that the divisor class $[\partial X^w\cap X_v^{\circ}] = [-K_{X^w\cap X_v^{\circ}}]$. That is, $\partial X^w\cap X_v^{\circ}$ is an anticanonical divisor for $X^w\cap X_v^{\circ}$.

\begin{proposition}\label{Gorenstein-local}
The blow-up of $X^w$ along $\partial X^w$ is Gorenstein iff the blow-up of $X^w\cap X^{\circ}_v$ along $\partial X^w\cap X^{\circ}_v$ is Gorenstein for each $v\in W$.
\end{proposition}
\begin{proof}

Being Gorenstein is a local property, so the result can be checked on the open sets $X^w\cap vX^{\circ}_{\id}$. By the Kazhdan-Lusztig Lemma, we can restrict to checking $(X^w\cap X^{\circ}_v)\times \mathbb{A}^{l(v)}$. Since the blow-up of $(X^w\cap X^{\circ}_v)\times \mathbb{A}^{l(v)}$ along $(\partial X^w\cap X_v^{\circ})\times \mathbb{A}^{l(v)}$ is just $\widetilde{(X^w\cap X^{\circ}_v)}\times \mathbb{A}^{l(v)}$ (ie. the product of the total transform of $X^w\cap X_v^{\circ}$ with affine space), $\widetilde{X^w}$ is Gorenstein iff $\widetilde{(X^w\cap X^{\circ}_v)}$ is Gorenstein, and the result follows.
\end{proof}

\begin{example}\label{KL-example}
\normalfont Let us compute local equations for the variety $X^{53241}$ at $e_{\id}$. To do this we need to intersect the Schubert conditions from the permutation matrix

 \[ \left( \begin{array}{ccccc}
 0 & 0 & 0 & 0 & 1\\
 0 & 0 & 1 & 0 & 0 \\
 0 & 1 & 0 & 0 & 0 \\
 0 & 0 & 0 & 1 & 0 \\
 1 & 0 & 0 & 0 & 0\end{array} \right)\]
(which provide defining equations for $\pi^{-1}(X^{53241})$ where $\pi:G\rightarrow G/B$) with the coordinates from the open cell (which is isomorphic to $\mathbb{A}^{10}$)

 \[ \left( \begin{array}{ccccc}
 1 & 0 & 0 & 0 & 0\\
 z_{41} & 1 & 0 & 0 & 0 \\
 z_{31} & z_{32} & 1 & 0 & 0\\
 z_{21} & z_{22} & z_{23} & 1 & 0 \\
 z_{11} & z_{12} & z_{13} &  z_{14} & 1\end{array} \right)\]

 The Schubert conditions say that the southwest-most $2\times 3$ matrix has rank $<2$. Therefore the Kazhdan-Lusztig ideal is generated by the three $2\times 2$ minors:

$$I_{w,v} = \langle z_{11}z_{22} - z_{12}z_{21}, z_{11}z_{23} - z_{13}z_{21}, z_{12}z_{23} - z_{13}z_{22}\rangle.$$\qed
 \end{example}

\subsection{The Gorenstein property}\label{gorenstein}

The Gorenstein property is defined using the canonical bundle, so we must take care in discussing this property for singular varieties. Let $X$ be a normal variety of dimension $n$ and $U$ its regular locus. We can define the sheaf of differential forms of degree $n$ on $U$ without issues. This sheaf is invertible, and hence is of the form $\mathcal{O}_U(D)$ for some Cartier divisor $D$ on $U$. Since $\codim (X\setminus U) \geq 2$, by normality we can extend $D$ to a Weil divisor on all of $X$ which we call a canonical divisor $K_X$. A similar definition works for the anticanonical divisor. These divisors are unique up to linear equivalence.

\begin{definition}
An normal algebraic variety $X$ is Gorenstein if $\omega_X$ is invertible. That is, $X$ is Gorenstein if $K_X$ is a Cartier divisor.
\end{definition}

Since Schubert varieties are normal (see Lemma \ref{schubert-normal}), Kazhdan-Lusztig varieties are also normal, so this definition makes sense for our purposes.

\begin{example}\normalfont
In the $GL_n(\mathbb{C})$ case we can detect which $X^w$ are Gorenstein using interval pattern avoidance (see \cite{Woo-Yong}). For $n\leq 5$, all $X^w$ are Gorenstein except for

$$X^{53241}\text{, } X^{35142}\text{, } X^{42513}\text{, } X^{52431}.$$\qed
\end{example}

Frobenius splittings will be heavily used in later arguments which is why we  work with anticanonical divisors instead of canonical ones. This still works for the purposes of detecting the Gorenstein property. We thank Michel Brion for the discussion regarding this topic.

\begin{lemma}\label{anti-gorenstein}
Let $X$ be a normal algebraic variety such that the anticanonical divisor $-K_X$ is Cartier. Then $X$ is Gorenstein.
\end{lemma}
\begin{proof}
If $-K_X$ is Cartier, then there exists a collection $\{(U_i,f_i)\}$ where the open subsets $U_i$ cover $X$ and $-K_X$ is the divisor of the rational function $f_i$. Then $K_X$ is the divisor associated to $\{(U_i,\frac{1}{f_i})\}$ and is also Cartier.
\end{proof}

\subsection{The subword complex}\label{subword}

We will see in subsequent sections that a Kazhdan-Lusztig variety $X^w\cap X_v^{\circ}$ can be degenerated to a Stanley-Reisner scheme associated to a subword complex. Subword complexes were first introduced by Knutson and Miller in \cite{KM2}. Given an ordered list $Q$ of simple reflections in a Coxeter group $\Pi$ and a fixed element $\pi\in \Pi$, one can ask about what structure can be placed on collection of subwords $Q$ that are also reduced expressions for $\pi$.

\begin{definition}
Given a word $Q = (\sigma_1,...,\sigma_m)$ where $\sigma_i$ a simple reflection in $\Pi$ and some fixed element $\pi\in\Pi$, we can define $\Delta(Q,\pi)$ to be the set of subwords $Q\setminus P$  where $P$ is a subword that contains a subsequence which is a reduced expression for $\pi$.
\end{definition}

The subword complex has some useful properties that are worth mentioning here (although we will only be concerned with the case where $\Pi$ is a Weyl group).

\begin{proposition}
$\Delta(Q,\pi)$ is a pure simplicial complex whose facets are the subwords $Q\setminus P$ such that $P\subset Q$ represents $\pi$.
\end{proposition}
\begin{proof}
See Lemma 2.2 in \cite{KM}.
\end{proof}

\begin{proposition}
The subword complex $\Delta(Q,\pi)$ is either a ball or sphere. A face $Q\setminus P$ is in the boundary of $\Delta(Q,\pi)$ iff $P$ has Demazure product $\neq \pi$.
\end{proposition}
\begin{proof}
See Theorem 3.7 in \cite{KM}.
\end{proof}

\begin{example}\normalfont
Let $\Pi =S_4$ and consider $\Delta(s_3s_2s_3s_2s_3,1432)$ where $s_i$ is the simple reflection swapping $i$ and $i+1$ and $\pi=1432$ is a permutation. Here $\pi$ has two reduced expressions in terms of the $s_i$, namely $s_3s_2s_3$ and $s_2s_3s_2$.
The facets will correspond to subwords $P$ of $Q$ such that $Q\setminus P$ is a reduced expression for $\pi$. For example, the subsequence chosen as the first and second elements of $Q$ is denoted by $s_3s_2---$ in the diagram below. Its complement is $--s_3s_2s_3$ which provides a reduced expression for $\pi$.

\begin{figure}[h]
\centering
\includegraphics[width=0.45\linewidth]{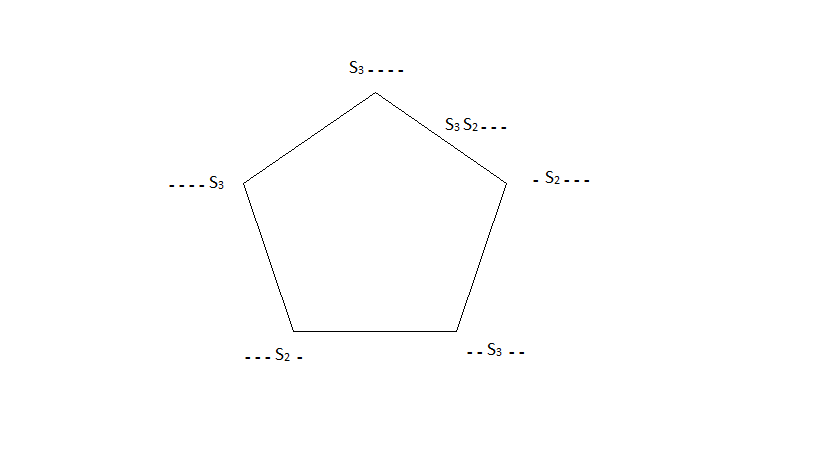}
\caption{}
\label{}
\end{figure}
\qed
\end{example}

\subsection{Frobenius splittings}\label{frobenius}
As motivation for the definition below, consider a commutative ring $R$, and let us observe that $R$ is reduced if the map $x\rightarrow x^n$ only sends $0$ to $0$ for $n>0$. We would like to write this condition as $\ker(x\rightarrow x^n)  =0$, but we cannot because the map is not linear in general. It does however make sense if we restrict to the case that $n$ is prime and $R$ contains the field $\mathbb{F}_p$, i.e. the map is the Frobenius endomorphism. Then, $R$ being reduced says that there exists a one-sided inverse to the Frobenius map.

A \textbf{Frobenius splitting} of an $\mathbb{F}_p$-algebra $R$ is a map $\varphi: R\rightarrow R$ which satisfies:\newline
\begin{align*}
&(i) \text{ } \varphi(a+b) = \varphi(a)+\varphi(b)\\
&(ii) \text{ }\varphi(a^pb) = a\varphi(b)\\
&(iii) \text{ }\varphi(1)=1\newline
\end{align*}

An ideal $I\subset R$ is compatibly split if $\varphi(I)\subset I$. See \cite{Knut} for a more detailed exposition on the subject.

\subsubsection{The trace map}\label{trace}
Let $R= \mathbb{F}_p[x_1,...,x_n]$. A simple example of a splitting on $R$ is the \textbf{standard splitting} defined on monomials $m$ (and extended linearly). It is the $p^{th}$ root map when $m$ is a $p^{th}$ power and 0 otherwise.
The ideals that are compatibly split by the standard splitting are precisely the Stanley-Reisner ideals.

More generally define the trace map $\tr(\cdot)$ first on monomials $m$ (and then extend linearly):
\[
\tr(m) =
  \begin{cases}
      \hfill  \frac{\sqrt[p]{m\prod_i x_i}}{\prod_i x_i}    \hfill & \text{ if $m\prod_i x_i$ is a $p^{th}$ power} \\
      \hfill 0 \hfill & \text{ otherwise} \\
  \end{cases}
\]

In general, given $f\in R$, $\varphi(g) =\tr(f^{p-1}g)$ defines a near splitting (taking $f= \prod_i x_i$ gives the standard splitting). If in addition we have $\tr(f^{p-1})=1$, then we have a Frobenius splitting. For example, $xyz$ always defines a splitting on $\mathbb{F}_p[x,y,z]$ and $x^2y+z^3+y^3+w^7$ defines a splitting on $\mathbb{F}_7[w,x,y,z]$.

\subsubsection{Frobenius splittings for schemes}\label{frobenius-scheme}

More generally, let $X$ be scheme (separated of finite type over an algebraically closed field of characteristic $p>0$). The absolute Frobenius morphism $F_X: X\rightarrow X$ is the identity on $X$ and the $p^{th}$ power map on $\mathcal{O}_X$. We say that $X$ is Frobenius split if the $\mathcal{O}_X$-linear map $F^{\#} : \mathcal{O}_X \rightarrow F_{*}\mathcal{O}_X$ splits (i.e. there exists a map $\varphi $ such that $\varphi\circ F^{\#}$ is the identity map).

\begin{lemma}
Frobenius split schemes are reduced.
\end{lemma}
\begin{proof}
See Proposition 1.2.1 in \cite{Brion-Kumar}.
\end{proof}

Just as we used $(p-1)^{st}$ powers of $f\in R[x_1,...,x_n]$ to define splittings on $\mathbb{A}^n$, we can define Frobenius splittings in general using $(p-1)^{st}$ powers of sections of the anticanonical bundle of $X$. For example, nonsingular projective irreducible curves of genus $g\geq 2$ are not split. See Sections 1.3 and 1.4 in \cite{Brion-Kumar} for more details regarding spitting using $\omega_X^{-1}$ and refer to Section 1.6 for instruction on how to pass from characteristic $p$ statements to characteristic 0 results.

\begin{proposition}\label{schubert-split}
There exists a Frobenius splitting of $G/B$ which compatibly splits all Schubert subvarieties $X^w$ and and opposite Schubert subvarieties $X_w$.
\end{proposition}
\begin{proof}
See Section 2.3 in \cite{Brion-Kumar}.
\end{proof}

\begin{corollary}
There is a Frobenius splitting of the Kazhdan-Lusztig variety $X^w\cap X_v^{\circ}$ which compatibly splits $\partial X^w\cap X_v^{\circ}$.
\end{corollary}
\begin{proof}
The intersection and union of compatibly split subvarieties is again compatibly split, so the result follows by Proposition \ref{schubert-split}.
\end{proof}

\subsection{Bott-Samelson resolutions}\label{BSR}

Let $Q=(w_1,...,w_k)$, where $w_i\in W$. A generalized Bott-Samelson variety $BS^Q$ is the quotient of $\overline{Bw_1B}\times ...\times \overline{Bw_kB}$ by the $B^k$ action given by:

$$(b_1,...,b_k)\cdot (p_1,...,p_k):= (p_1b_1^{-1},b_1p_2b_2^{-1},...,b_{k-1}p_kb_k^{-1}).$$

If we let $\times^B$ denote the quotient by the above action for the $k=2$ case, we can write $BS^Q$ as $\overline{Bw_1B}\times^B ...\times^B \overline{Bw_kB}/B$.

We can define Dem(Q) inductively as

\[
\text{Dem((Q,$s_i$))} = \begin{cases}
\text{Dem(Q)$\cdot s_i$} & l(\text{Dem(Q$\cdot s_i$)}) > l(\text{Dem(Q)}) \\
\text{Dem(Q)} & \text{otherwise} \\
\end{cases}\]
where the Demazure product of the empty word is the identity. Then the $B$-equivariant map

$$\varphi_Q:\overline{Bw_1B}\times^B ...\times^B \overline{Bw_kB}/B \rightarrow G/B$$
defined by $(p_1,...,p_k) \rightarrow p_1\cdot ... \cdot p_k B/B$ maps onto $X^{\text{Dem(Q)}}$.

Sub Bott-Samelson varieties $BS^R\subset BS^{Q}$ are naturally defined by taking $R=(v_1,...,v_k)$ where $v_i\leq w_i$.

\begin{lemma}\label{BS-smooth}
Let $Q=(w_1,...,w_k)$ where $w_i$ is a simple reflection $s_{k_i}\in W$. Then $BS^Q$ is smooth.
\end{lemma}
\begin{proof}
As seen in \cite{DaSilva-strict}: In this case, $\overline{Bw_iB}\cong P_{k_i}$ is a minimal parabolic subgroup of $G$. Then $BS^Q$ is just an iterated $\mathbb{P}^1$-bundle and $\varphi_Q$ defines a resolution of singularities for $X^w$.
\end{proof}

See \cite{Escobar} for an example of how to visualize the $BS^Q$ in Lemma \ref{BS-smooth} in the context of flag varieties. Although $BS^Q$ is not smooth in general, it is Cohen-Macaulay and normal.

\begin{lemma}
The Bott-Samelson variety $BS^Q$ is both Cohen-Macaulay and normal.
\end{lemma}

\begin{proof}
As seen in \cite{DaSilva-strict}: Using the projection map

$$\overline{Bw_1B}\times^B ...\times^B \overline{Bw_kB}/B \rightarrow \overline{Bw_1B}\times^B ...\times^B \overline{Bw_{k-1}B}/B $$
it is easy to see that $BS^Q$ is just an iterated Schubert variety bundle over $X^w$. It is well known that $X^w$ is both Cohen-Macaulay and normal. Since both of these conditions are local, and since the bundle map above is locally trivial, it follows by induction that $BS^Q$ also has these properties.
\end{proof}

\begin{lemma}\label{BS-birational}
 Let $\varphi_Q$ be a Bott-Samelson map with $Q=(w_1,...,w_k)$. Then $\varphi_Q$ is always a proper map. Furthermore, $\varphi_Q$ is birational if and only if $Dem(Q) = \prod w_i$ (in which case we say $Q$ is reduced).
\end{lemma}

\begin{proof}
As seen in \cite{DaSilva-strict}: Since $BS^Q$ and $X^w$ are projective, they are both proper over $\mathbb{C}$ so $\varphi_Q$ must be proper also. The second statement follows from the fact that $\varphi_Q$ is an isomorphism away from the sub Bott-Samelson varieties $BS^R$ where $R$ is not reduced. Then $Q$ is reduced iff $R\neq Q$.
\end{proof}

\subsection{Geometric vertex decomposition}\label{GVD}
Before we describe the degeneration of $X^w\cap X_v^{\circ}$ to the Stanley-Reisner scheme, let us define the \textbf{geometric vertex decomposition} (see the work of Knutson in  \cite{Knut-patches} and \cite{Knut}).

We start with an integral variety $X\subset \mathbb{A}^n$. Let us write $\mathbb{A}^n=H\times L$, with $H$ a hyperplane and $L$ a line. Consider the action of $\mathbb{G}_m$ on $\mathbb{A}^n$ by
$$z\cdot(x,l) = (x,zl)$$
and let us define

$$X'= \lim_{t\rightarrow 0} t\cdot X.$$

Here $X'$ is a limit scheme. That is, it is the zero fibre of

$$ \overline{\bigcup_{z\in\mathbb{G}_m} z\times (z\cdot X)} \subset \mathbb{A}^1\times (H\times L).$$

There is a simple way to compute $X'$ set-theoretically. Let $\Pi\subset H$ be the closure of the image of the projection of $X$ to $H$. We will denote the closure of $X$ in $H\times (L\cup\{\infty\})$ by $\bar{X}$. Finally, define $\Lambda\subset H$ by $\Lambda\times\{\infty\} = \bar{X}\cap(H\times\{\infty\}) = \bar{X}\setminus X$.

\begin{lemma}
The limit scheme $X'$ of the geometric vertex decomposition can be described as a set by
$$X' = (\Pi\times\{0\})\cup_{\Lambda\times 0} (\Lambda\times L).  $$
\end{lemma}

\begin{example}
Let us write $\mathbb{A}^2=H\times L$ where $H$ is the $x$-axis and $L$ is the $y$-axis. Consider the variety $X$ defined by $xy-1=0$.  The closure of the projection of $X$ onto $H$ is exactly $H$. Similarly, it is not hard to see that $\Lambda\times L$ is the $y$-axis, and by the Lemma, $X'$ is the union of the two axes.\qed
\end{example}

This definition generalizes if we take $H$ to be some general scheme, as long as $L$ remains $\mathbb{A}^1$.

The degeneration that appears in the next section can be shown inductively by using geometric vertex decomposition applied to a Kazhdan-Lusztig variety $KL_1$ where $H$ can be chosen so that $\Pi = KL_2$ and $\Lambda =KL_3$ are two other Kazhdan-Lusztig varieties. More precisely, let $X_w|_v := X_w \cap (vN_{-}B_+/B_+)$ where $vN_{-}B_+/B_+$ is the permuted big cell (called a Schubert patch on $X_w$).

\begin{theorem}[Theorem 2, \cite{Knut-patches}]
  Let $v,w$ be elements of $W$. If $X_w|_v \neq \emptyset$, then $v \geq w$ in the Bruhat order. Assume this hereafter. 
  
  If $v = 1$, then $w = 1$ and $X_w|_v = N_{-}B_+/B_+$. Otherwise, there exists a simple root $\alpha$ such that $vr_{\alpha} < v$ in the Bruhat order. Let $X'$ be the degeneration of $X_w|_v$ described above.
  
\begin{itemize}
\item If $wr_{\alpha} > w$, then $X' = X_w|_v$ (the limiting process is trivial), and 
$$X_w|_v \cong \Pi \times \mathbb{A}^1_{-v\cdot\alpha}\text{ , } X_w|_{vr_{\alpha}}\cong \Pi\times  \mathbb{A}^1_{v\cdot\alpha}$$
for the same $\Pi$.
\item If $wr_{\alpha} < w$ but $w \nleq vr_{\alpha}$, then $X' = X_w|_v$ (again, the limiting process is trivial), and $$X_{wr_{\alpha}}|_{vr_{\alpha}}\cong X_w|_v \times \mathbb{A}^1_{v\cdot\alpha}.$$ 
\item If $wr_{\alpha}< w \leq vr_{\alpha}$, then $X'$ is reduced, and has two components: $$X' = (\Pi \times {0}) \cup_{\Lambda\times \{0\}}(\Lambda\times \mathbb{A}^1_{- v\cdot\alpha})$$ where $\Pi \times \mathbb{A}^1_{v\cdot\alpha} \cong X_{wr_{\alpha}}|_{vr_{\alpha}}$ and $\Lambda \times \mathbb{A}^1_{v\cdot\alpha}\cong X_w|_{vr_{\alpha}}.$
\end{itemize}
\end{theorem}

Simply put, we can degenerate $KL_1$ as:

$$KL_1 \rightsquigarrow KL_2 \cup_{0\times KL_3} ( \mathbb{A}^1\times KL_3).$$
We want to emphasize that this limit is in fact reduced. This inductive version of the degeneration can be found in greater detail in \cite{Knut-patches} and \cite{Knut}.

\section{Degenerating to the toric case }\label{toric}

We will now consider the degeneration of the Kazhdan-Lusztig variety $X^w\cap X^{\circ}_v$ to a Stanley-Reisner scheme. To give a picture of what will occur later in this section, we provide a motivational example.

\begin{example}\label{degen-example}\normalfont

In the $GL_n(\mathbb{C})$ case, consider $X^{53241}\cap X_{12345}^{\circ}$. Then by Example \ref{KL-example} we have

$$I_{w,v} = \langle z_{11}z_{22} - z_{12}z_{21}, z_{11}z_{23} - z_{13}z_{21}, z_{12}z_{23} - z_{13}z_{22}\rangle$$

Consider the term order where $z_{11} > z_{21} > ... > z_{12} > ... > z_{55}$. Then

$$\init(I_{w,v}) = \langle z_{11}z_{22},  z_{11}z_{23}, z_{12}z_{23} \rangle$$

This is the Stanley-Reisner ideal of the subword complex $\Delta(Q,12345)$ (the vertex $z_{13}$ is a cone vertex and has not been drawn below).

\begin{figure}[H]
\centering
\includegraphics[width=0.5\linewidth]{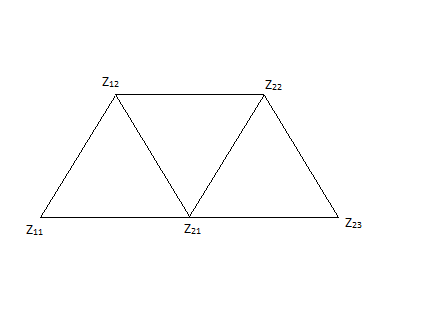}
\caption{}
\label{SR}
\end{figure}
Note: The fact that these ``Woo-Yong'' coordinates in \cite{Woo-Yong} are equivalent to Bott-Samelson coordinates in this case follows from \cite{Woo-Yong-Grobner}.\qed

\end{example}

\begin{theorem}\label{Knutson09}
 Given $\varphi_Q:BS^Q \rightarrow G/B$, consider $\varphi_Q^{-1}(X_w)\subset \mathbb{A}^{|Q|}$ and denote its ideal by $I_w$. Then $J_w: = \init(I_w)$ is a squarefree monomial ideal. Its prime components are coordinate ideals $C_F$ where $Q\setminus F$ is a reduced word for $w$.
\end{theorem}
\begin{proof}
See \cite[\S 7.3]{Knut}.
\end{proof}

This theorem is stated in terms of $X_w\cap X_{\circ}^v$ with $Q$ a reduced word for $v$. Multiplying by the long Weyl group element $w_0$ gives us our form of Kazhdan-Lusztig variety introduced in Section \ref{KL}. We therefore denote the simplicial complex associated to $I_{w,v}$ by $\Delta(Q,w)$ where $Q$ is reduced word for the permutation $v$. This is the \textbf{subword complex} introduced in Section \ref{subword}

\begin{theorem}\label{KnutsonMiller04}
The subword complex $\Delta(Q,w)$ is homeomorphic to a ball and is shellable.  Its boundary sphere is
$$\bigcup_{w'\gtrdot w} \Delta(Q,w').$$ We will denote its defining ideal by $\partial I_{w,v} := \cap_{w\gtrdot w'} I_{w',v}$.
\end{theorem}
\begin{proof}
See \cite{KM}.
\end{proof}

\begin{corollary}\label{Knutson-corollary}
Properties of $I_{w,v}$:
\begin{itemize}
\item Each $I_{w,v}$ is Cohen-Macaulay.
\item Each $I_{w,v}$ is normal and $\partial I_{w,v}$ defines an anitcanonical divisor.
\item $I_{w,v} = \cap_{w'\gtrdot w, \text{biGrassmannian}} I_{w',v}$. The concatenation of Grobner bases for the maximal biGrassmannians $\leq w$ provides a Grobner basis for $I_{w,v}$.
\item Consider the poset morphism $m: 2^Q\rightarrow W$ given by $F\rightarrow Dem(Q\setminus F)$. Then $\Delta(Q,w) = m^{-1}([w,w_0])$. That is, $m$ defines a ``Bruhat decomposition for $\Delta(Q,w)$''.
\end{itemize}
\end{corollary}
\begin{proof}
A combination of results from \cite{Knut} and \cite{KM}.
\end{proof}

\subsection{Combinatorics}\label{combinatorics}

We wish to understand what happens when we blow-up the toric scheme (coming from the degeneration) along its boundary.

First recall that a normal (irreducible) variety containing a torus $T\cong (\mathbb{G}_m)^n$ as an open subset, where the action of $T$ on itself extends to an action of $X$ is called a toric variety. The \textbf{boundary} of $X$ is then defined as $\partial X = X\setminus T$. Denote the coordinate ring of $X$ by $R[X] = k[x_1,...,x_n]/I_X$. To each affine toric variety $X$ we can associate a convex rational polyhedral cone $\sigma_X$ contained in the character lattice $M_{\mathbb{R}}$. This cone defines a saturated affine semigroup $S_{\sigma_X}$ such that $X = \spec(k[S_{\sigma_X}])$ (we are using the notation in \cite[\S 1]{CLS}). The irreducible components $D_1,...,D_n$ of $\partial X$  are defined as the vanishing of a primitive character in $M_{\mathbb{R}}$ and correspond to  the facets of $\sigma_X$. Here $\partial X$ is the anticanonical divisor (up to linear equivalence). See \cite{CLS} for more information about the correspondence between $X$ and a cone or polytope in the character lattice. 

A blow-up of an affine $X$ along a (reduced) torus invariant subvariety of $\partial X$ can be viewed in the character lattice as ``planing off'' the corresponding faces of $\sigma_X$. That is, it can viewed by taking the convex hull of the lattice points that remain after removing components of the boundary. For now we will consider the combinatorics of such an operation and save the geometric questions for the next section.

We will be working with reducible toric schemes which can be associated to convex polytopal complexes. A \textbf{lattice polytopal complex} $\Delta$ is a union of lattice polytopes such that the intersection of any two is also a lattice polytope contained in $\Delta$. The boundary $\partial\Delta$ of $\Delta$ is the union of polytopes which form the boundary of the convex region $\Delta$. Affine toric schemes can then be associated to the cone on $\Delta$, denoted by $\cone(\Delta)$ or $C(\Delta)$.

Just as in the irreducible case, the \textbf{blow-up of a lattice polytopal complex} $\Delta$ along a subcomplex $\Delta'\subset \partial \Delta$ is the polytopal complex $\widetilde{\Delta}$ resulting from taking the convex hull of the lattice points in $\Delta\setminus \Delta'$. We will show that in the character lattice $M_{\mathbb{R}}$, this is the correct picture to visualize the Rees algebra $R[It]=\bigoplus_{k\geq 0} I^kt^k$ for the blow-up of $X$ along the reduced subscheme $(\partial X)'$ (corresponding to $\Delta$ and $\Delta'$ respectively).

\begin{lemma}\label{BU}
Let $X$ be an affine toric scheme associated to a polytopal complex $\Delta\subset M_{\mathbb{R}}$ (so that $X=\spec(k[S_{C(\Delta)}]))$. Then the blow-up of $X$ along a torus invariant subscheme $(\partial X)'\subset\partial X$ is again a toric scheme that can be associated to the blow-up of $\Delta$ along $\Delta'$.
\end{lemma}
\begin{proof}
Let $I$ be the defining ideal for $(\partial X)'$. The blow-up algebra $R[It]$ is generated by $I$ in degree one over $R[X]$. In the character lattice, $R[X]$ can be viewed as $C(\Delta)$ and $I$ as the subset of lattice points not in $C(\Delta ')$ (we choose an appropriate embedding so that the lattice is dense enough). We observe that the weights of $I^k$ scaled by $\frac{1}{k}$ form the same cone $C(\widetilde{\Delta})$ in degree one, except using a denser lattice. Therefore it suffices to compute $\widetilde{X}$ using the weights in the blow-up of $\Delta$ along $\Delta'$. In fact $\widetilde{X} = \proj (k[S_{\widetilde{C(\Delta)}}])$.
\end{proof}
After the first blow-up, our scheme is no longer affine. To continue viewing the blow-up as in Lemma \ref{BU}, we just need to cover the total transform with affine charts and patch together the results.

\begin{example}\label{BU_origin}\normalfont
The blow-up of $\mathbb{C}^2$ at the origin can viewed in the character lattice as the first region in the figure below. The origin of $\mathbb{C}^2$ corresponds to the origin of the first quadrant. After removing it and taking the convex hull of the lattice points that remain, we are left with the first picture below corresponding to $I =\langle x,y \rangle$. We can obtain the diagram for $I^2$ by planing off the diagonal an extra step inward (the second diagram below). This new diagram scaled by $\frac{1}{2}$ is the original one with a denser lattice.
\begin{figure}[H]
\centering
\label{BU_origin_figure}
\includegraphics[width=0.5\linewidth]{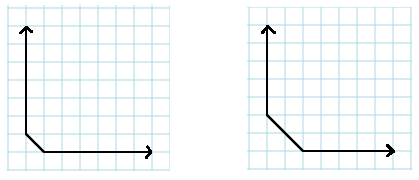}
\caption{}
\qed
\end{figure}

\end{example}
Given the cone on $\Delta$, let $\pi^1(C(\Delta)) =\conv(C(\Delta)\setminus \partial C(\Delta))$. We can perform the same operation on $\pi^1(C(\Delta))$ by taking the convex hull after removing the boundary, and we will denote this by $\pi^2(C(\Delta))$. Continuing this process, we have the following sequence:

$$... \rightarrow \pi^3(C(\Delta))\rightarrow \pi^2(C(\Delta))\rightarrow \pi^1(C(\Delta))\rightarrow \pi^0(C(\Delta))=C(\Delta)$$

We will say that $C(\Delta)$ is \textbf{stable} if there exists a $k>0$ such that the toric scheme associated to $\pi^k(C(\Delta))$ is isomorphic to the toric scheme associated to $\pi^{k-1}(C(\Delta))$. In this case we will simply say that $\pi^k(C(\Delta))$ and $\pi^{k-1}(C(\Delta))$ are \textbf{equivalent}.

The most basic simplicial complex, just a single standard simplex, stabilizes very quickly, even if we only remove a portion of the boundary. This single example is the only case we actually need to extend the result to Stanley-Reisner complexes.

\begin{lemma}\label{BU_simplex}
Let $\Delta$ be the standard simplex in $\mathbb{A}^n$. Let $\Delta'$ be some proper simplicial subcomplex of $\Delta$. Let $\widetilde{\Delta} = \conv(C(\Delta)\setminus C(\Delta'))$ and let $\widetilde{\Delta'}\subset \widetilde{\Delta} $ be the exceptional facets (those facets in $\widetilde{\Delta}$ not originally in $\Delta$). Then the blow-up of $\widetilde{\Delta}$ along $\widetilde{\Delta'}$ is equivalent to $\widetilde{\Delta}$.
\end{lemma}
\begin{proof}


First consider the case that $\Delta'$ is just one face of $\Delta$. We may assume by a change of coordinates that it is defined by
$$\{x_1=...=x_k=0\}\cap \Delta \text{ for } 1<k< n$$

Taking the convex hull of $C(\Delta)\setminus C(\Delta')$ is equivalent to intersecting $C(\Delta)$ with the half space $H_1^+ = \{x_1+...+x_k\geq 1\}$ (since this has removed the portion where $\{x_1=...=x_k=0\}$ leaving only lattice points where at least one $x_i>0$). The new facet of the boundary $\widetilde{\Delta'}$ consists of the polytope $C(\Delta) \cap \{x_1+...+x_k= 1\}$.

Blowing-up once more would result in a region given by intersecting $C(\Delta)$ with half space $H_2^+ = \{x_1+...+x_k\geq 2\}$ and boundary defined by $x_1+...+x_k=2$.

It is clear that both correspond to isomorphic copies (using a Veronese map giving different embeddings) of the blow-up of $\mathbb{A}^n$ along $\{x_1=...=x_k=0\}$, hence we have stability.

The case for general $\Delta'$ easily follows.
\end{proof}

The simplicial complexes we will be concerned with are those that come from Stanley-Reisner ideals. A \textbf{Stanley-Reisner complex} is a simplicial complex whose face ideal is generated by square-free monomials. The advantage to working with such complexes is that their simplices are stable.

\begin{lemma}\label{SR_standard}
Let $I \subset k[x_1,...,x_n]$ be an ideal generated by square-free monomials. Then the affine semigroup $S_{C(\Delta_I)}$ in the character lattice $M_{\mathbb{R}}$ associated to the reducible affine toric scheme $$X=\spec( k[x_1,...,x_n]/I)$$ is the cone on the Stanley-Reisner complex $\Delta_I$. Furthermore, each simplicial cone in $C(\Delta_I)$ is isomorphic to the cone on a standard simplex.
\end{lemma}
\begin{proof}
Since $I$ is a square free monomial ideal, each component of $I$ corresponds to a coordinate subspace of $\mathbb{A}^n$. It is an easy exercise to check that the cones in $M_{\mathbb{R}}$ corresponding to each component are isomorphic to some orthant $\mathbb{A}^k$, which is the cone on some standard simplex in the appropriate coordinates. The gluing of these simplices must occur along faces which are not in the face ideal $I$.

\end{proof}

\begin{example}\label{SR-example}\normalfont
Consider $\mathbb{C}[w,x,y,z]/\langle wz\rangle$. There are two components in this affine scheme defined by $w=0$ and $z=0$. The $w=0$ component is given by $\spec(\mathbb{C}[x,y,z])$ which corresponds to the orthant $\mathbb{C}_{\geq 0}^3$ in the character lattice. A similar statement holds for the $z=0$ component. Since $wz=0$, adding the generator of the $w$-axis and the generator for the $z$-axis results in 0, a relation demonstrated by the gluing of one facet from each simplex (corresponding to the gluing along the $xy$-plane).

\begin{figure}[H]
\centering
\label{BU_origin_figure}
\includegraphics[width=0.37\linewidth]{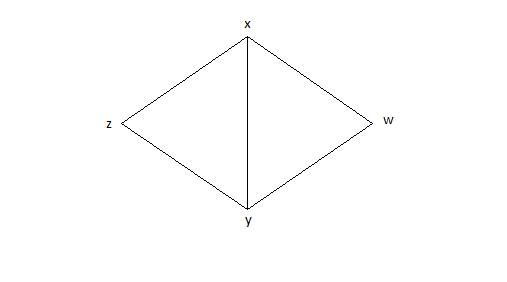}
\caption{}
\qed
\end{figure}

\end{example}

\begin{theorem}\label{stable}
Let $\Delta_I$ be as in Lemma \ref{SR_standard}. Then the blow-up of $\Delta_I$ along $\partial\Delta_I$ is stable. Even more, $\pi^2(C(\Delta_I))=\pi^1(C(\Delta_I))$.
\end{theorem}
\begin{proof}
It suffices to show that each simplex $\sigma\in \Delta_I$ blown up along $\sigma \cap \partial \Delta_I$ is stable. By Lemma \ref{SR_standard}, each simplex $\sigma$ is isomorphic to the standard simplex. Therefore the result holds by Lemma \ref{BU_simplex}.
\end{proof}

\begin{example}\normalfont
We observe that the subword complex from Example \ref{degen-example} is stable. Removing the boundary and taking the convex hull would produce the blue polytopal complex. It is clear that doing this process once more produces an isomorphism.

\begin{figure}[H]
\centering
\includegraphics[width=0.5\linewidth]{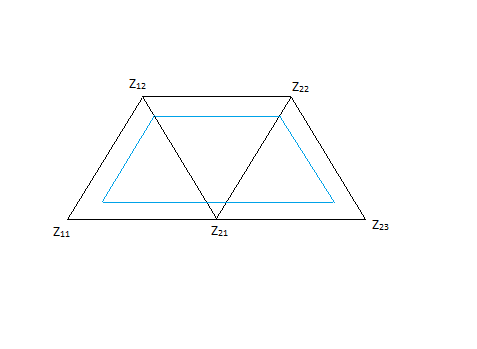}
\caption{}
\label{}
\end{figure}

\qed
\end{example}

\subsection{Boundary divisors and Gorenstein toric varieties}

In Section \ref{combinatorics}, we showed that blowing-up a Stanley-Reisner scheme along the boundary twice yields an isomorphism. Let us recall the following fact that follows immediately from the universal property of blow-ups (see Section \ref{universal}):

\begin{lemma}\label{BU-iso}
The blow-up of a scheme $X$ along some closed subscheme $Y$ is an isomorphism iff $Y$ is an effective Cartier divisor of $X$.
\end{lemma}

This means that the (reduced) boundary divisor is Cartier after the first blow-up in the sequence described above.

\begin{proposition}\label{SR-cartier}
Let $X$ be a Stanley-Reisner scheme, and consider the sequence

$$ X=X_0\stackrel{\pi_0}{\leftarrow} X_1 \stackrel{\pi_1}{\leftarrow}X_2 $$
where $\pi_i$ is the blow-up of $X_i$ along the reduction of the boundary $\partial X_i$ and $X_1$ and $X_2$ are the total transforms. Then $\partial X_1$ is an effective Cartier divisor.
\end{proposition}
\begin{proof}
By Theorem \ref{stable}, the polytopal complexes $\Delta_1$ and $\Delta_2$ associated to $X_1$ and $X_2$ are equivalent. Therefore $X_1$ is isomorphic to $X_2$. By Lemma \ref{BU-iso}, $\partial X_1$ is a Cartier divisor.
\end{proof}

Note that the process in Section \ref{combinatorics} defines the boundary $\partial X_1$ as the reduction of the exceptional divisor. The components of the exceptional divisor could a priori have multiplicity $>1$, which means that the blow-up along $\partial X_1$ would correspond to a weighted removal of the boundary followed by taking the convex hull of what remains. While the exceptional divisor is Cartier, its reduction might not be, and in general this might not produce an isomorphism like in Proposition \ref{SR-cartier}.

In the Stanley-Reisner case however, the reduced boundary divisor of the total transform is Cartier, but it is not immediately clear that it is \textit{the} boundary divisor in the sense that we want it to be. That is, we do not know whether $\partial X_1$ is anticanonical. We will use Frobenius splittings to guide our understanding of anticanonical for reducible schemes. First note that irreducible normal toric varieties are easy to Frobenius split.

\begin{lemma}\label{torus-invariant}
Let $X$ be a toric variety (irreducible and normal). Let $t_1,...,t_n$ be the coordinates on the torus $T$ coming from $\mathbb{G}_m$. Then

$$\sigma = \frac{dt_1\wedge ...\wedge dt_n}{t_1...t_n}  $$
is a rational section of $\omega_X$ and $\sigma^{1-p}$ defines the unique $T$-invariant splitting of $X$ that compatibly splits $\partial X$.

\end{lemma}
\begin{proof}
See Exercise 1.3.6 in \cite{Brion-Kumar}.
\end{proof}

In fact weakly normal toric schemes (see Section \ref{weakly_normal}) are also Frobenius split by the standard splitting. The next proposition shows that the total transform of the blow-up along the boundary is again Frobenius split. 

\begin{proposition}\label{toric-frobenius}
Let $X$ be a weakly normal affine toric scheme. Then the blow-up of $X$ along $\partial X$ is also Frobenius split.
\end{proposition}
\begin{proof}
Weakly normal toric schemes are Frobenius split by the standard splitting. Since $X$ is affine, there is a polytopal complex $\Delta$ such that $\cone(\Delta)$ defines an affine semigroup in the character lattice $M_{\mathbb{R}}$ associated to $X$.

One can check that $X$ is Frobenius split by the standard splitting iff the following condition holds:

$$(pa_1,...,pa_n)\in \cone(\Delta)\cap\mathbb{Z}^n \Rightarrow (a_1,...,a_n)\in \cone(\Delta)\cap\mathbb{Z}^n $$

In other words, the standard splitting acts as the multiplication by $\frac{1}{p}$ map on lattice points.

The blow-up of $X$ along $\partial X$ can be viewed in the character lattice by taking the convex hull of the lattice points in $\cone(\Delta)\setminus\cone(\partial(\Delta))$ by Section \ref{combinatorics}. Suppose the total transform was not Frobenius split by the standard splitting. Then there would be a lattice points $L = (pa_1,...,pa_n)\in \widetilde{\cone(\Delta)} $ such that $ \frac{1}{p}L \notin \widetilde{\cone(\Delta)}$. Since $\frac{1}{p}L \in \cone(\Delta)$, we conclude that $\frac{1}{p}L \in \cone(\partial(\Delta))$. But then $L \in \cone(\partial(\Delta))$ (since it is a cone) which means that $L$ couldn't be in $\widetilde{\cone(\Delta)}$, a contradiction.
\end{proof}

\begin{corollary}
Let $X$ be a Stanley-Reisner scheme. Then $X$ is Frobenius split by the standard splitting.
\end{corollary}
\begin{proof}
The components of $X$ and the components of their intersections are coordinate subspaces, so each component is normal and each split by the standard splitting.
\end{proof}

In the (irreducible) toric variety case, the blow-up along the boundary is Frobenius split by the standard splitting. By Section \ref{frobenius}, the standard splitting is defined using the unique toric invariant section of the anticanonical bundle, which has to be the boundary divisor. In this way we know that the boundary of the total transform is anticanonical.

We are dealing with reducible toric schemes however, so we need to take greater care in how we show that the boundary divisor is anticanonical. We have shown that the total transform of a Frobenius split toric variety is again split. This Frobenius splitting however does not necessarily come from a section of the anticanonical bundle as in Section \ref{frobenius}. At this point we need to decide how we want to define the anticanonical divisor for reducible schemes.

As motivation, recall that by Theorem \ref{KnutsonMiller04} we know that the $\partial X^w\cap X_{\circ}^v$ degenerates to the boundary of a subword complex. We would like to say that the anticanonical divisor $\partial X^w\cap X_{\circ}^v$ has degenerated to an anticanonical divisor. Consider Example \ref{SR-example}. In this case, we don't want to include the interior facet corresponding to the gluing of two orthants along the coordinate subspace $w=z=0$. Let us also adopt the philosophy that anticanonical sections should tell us what is Frobenius split. We know that $w=z=0$ is Frobenius split if $w=0$ and $z=0$ are, further showing that interior components used for gluing shouldn't be included in our definition of anticanonical.

By Section \ref{gorenstein}, we know how to define the anticanonical bundle for normal varieties. Given a reduced equidimensional algebraic variety $X$, let us denote its normalization by $\nu:\bar{X}\rightarrow X$. In this normalization process, irreducible components of $X$ have become disjoint. Suppose also that the irreducible components of $X$ are normal. Then the ramification locus $R\subset \bar{X}$ of $\nu$ is just the fibre over the intersection of the components (in our previous example, $w=z=0$). This ramification locus together with the components coming from the boundary of $X$ should be anticanonical (in the usual sense) in $\bar{X}$. This motivates the following definition:

\begin{definition}\label{anticanonical_reducible}
Let $X$ be a reduced equidimensional scheme whose irreducible components are normal. Let $\nu:\bar{X}\rightarrow X$ be its normalization and $R\subset \bar{X}$ the ramification locus. A divisor $D\subset X$ is called \textbf{anticanonical} if:
\begin{itemize}
\item $(\nu^{-1}(D)\cup R)\cap \bar{X}_{reg}$ is anticanonical in $\bar{X}_{reg}$\\
\item $\codim(\nu^{-1}(D)\cap R)>1$.
\end{itemize}
\end{definition}

The last condition ensures that invariant toric divisors in a toric scheme $X$ are either part of the branch locus or part of the anticanonical divisor of $X$ (indeed it forces the branch locus to be a divisor). We get the following immediate results from the definition.

\begin{lemma}\label{boundary-anticanonical}
Let $X$ be a reduced equidimensional toric scheme associated to the polytopal complex $\Delta$. Suppose that $\Delta$ is homeomorphic to a ball or sphere. Also assume that the components of $X$ are normal. Then $\partial \Delta$ is anticanonical.
\end{lemma}
\begin{proof}
Let $\nu: \bar{X}\rightarrow X$ be the normalization map. Since the components of $X$ are normal, the ramification locus $R$ of $\nu$ is just the fibre over facets of $\Delta\setminus\partial\Delta$. Then $R\cup\nu^{-1}(\partial\Delta)$ is just the union of the toric invariant divisors corresponding to the boundary of each irreducible component, so is anticanonical in $\bar{X}$ by \ref{torus-invariant}.

Since $\Delta$ is a ball or sphere, the intersection of the branch locus and $D$ has codimension $>1$.
\end{proof}

\begin{corollary}\label{subword-anticanonical}
Let $X$ be an affine toric scheme associated to the subword complex $\Delta = \Delta(Q,w)$. Then $\partial X$ is anticanonical. Furthermore $X$ is Gorenstein iff the blow-up along the $\partial X$ is an isomorphism.
\end{corollary}
\begin{proof}
By Section \ref{subword}, $\Delta$ is a ball or sphere. By Lemma \ref{boundary-anticanonical}, $\partial X$ is anticanonical. If $X$ is Gorenstein then $\partial X$ is a Cartier divisor which means the blow-up is an isomorphism by Lemma \ref{BU-iso}. Conversely if the blow-up is an isomorphism, then $\partial X$ is a Cartier divisor and $X$ is Gorenstein.
\end{proof}

The corollary shows that the simplicial complex from Example \ref{degen-example} is not Gorenstein. Gorenstein simplicial complexes have been studied in general, and necessary and sufficient conditions for determining when a Stanley-Reisner complex is Gorenstein can be found in \cite{Notbohm} and \cite{Stanley}. These conditions are formulated using the homology of links of faces in the complex. While this description is useful in a more general setting, it does not indicate when general polytopal complexes are Gorenstein, which is necessary for dealing with the blow-up of the Stanley-Reisner complex.

\begin{proposition}\label{SR-anticanonical}
Let $X$ be as in Lemma \ref{subword-anticanonical}. Then the boundary divisor of the blow-up of $X$ along $\partial X$ is an anticanonical divisor.
\end{proposition}
\begin{proof}
Each irreducible component in the total transform remains normal. Indeed the affine semigroup $S$ in the character lattice corresponding to each component remains saturated. Each component was a coordinate subspace, and their boundaries and intersections were again coordinate subspaces.

It is clear that the blow-up of $\Delta$ along $\partial \Delta$ remains a ball or sphere.
\end{proof}

\begin{corollary}\label{BU-gorenstein}
Let $X$ be a Stanley-Reisner scheme associated to a subword complex $\Delta = \Delta(Q,w)$. The blow-up of $X$ along its boundary $\partial X$ is Gorenstein.
\end{corollary}
\begin{proof}
By Proposition \ref{SR-anticanonical}, the boundary divisor of the blow-up is an anticanonical divisor. By Proposition \ref{SR-cartier}, it is also an effective Cartier divisor. By Lemma \ref{anti-gorenstein}, the blow-up is Gorenstein.
\end{proof}

These results hold more generally for Stanley-Reisner complexes which are homeomorphic to a ball or sphere (although we are only concerned with the ones which are also subword complexes). 

We have shown that the blow-up of the degeneration is Gorenstein. What remains is to show that the blow-up of $X^w\cap X_v^{\circ}$ along $\partial X^w\cap X_v^{\circ}$ can be degenerated to this Gorenstein variety. This is the content of Section \ref{degen-commute}.

\subsection{Combinatorial description of exceptional components}\label{SR_Exceptional}

When we blow-up a Stanley-Reisner scheme along its boundary, we get an exceptional divisor whose components can be seen in the blow-up of the subword complex. We wish to give a combinatorial description for any new exceptional components. First we need to understand which faces in a simplicial complex give rise to new facets in the boundary of the blow-up.

Let $\Delta$ be a simplicial complex homeomorphic to a ball with facets $\sigma_1,...,\sigma_n$. Let us define $\tau_i$ as the simplicial complex formed by intersecting $\sigma_i$ with the boundary $\partial\Delta$ and throwing away any facets of $\sigma_i$. More precisely, if $\mu_{i,1}, ... ,\mu_ {i,n_i}$ are the facets of $\sigma_i$, then define $S_i=\{j|\mu_{i,j} \in \sigma_i\cap\partial \Delta\}$ and let

\[ \tau_i =\displaystyle\cl((\sigma_i\cap\partial \Delta)\setminus \bigcup_{j\in S}\cl(\mu_{i,j}))\]
where $\cl(\cdot)$ is the closure of a set of simplices.

We will say that a facet $F$ of $\partial\widetilde{\Delta}$ whose image under the blow-up map is contained in the exceptional set of $\Delta$ is a \textbf{new facet} (seen below in green).

\begin{figure}[h]
\centering
\includegraphics[width=0.3\linewidth]{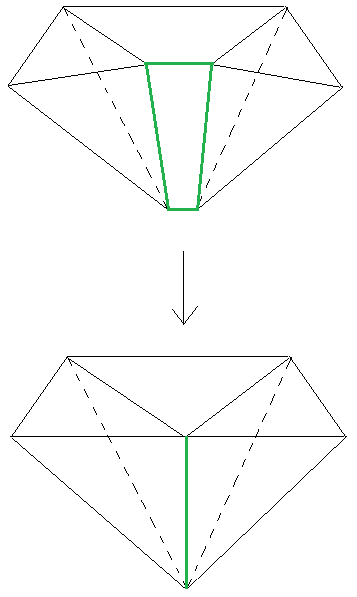}
\caption{}
\label{NewFace_3}
\end{figure}

 It is fairly easy to see that a maximal face of $\tau_i$ will produce a new facet in the boundary of the blow-up. Unfortunately these are not the only faces of $\tau_i$ which have this property.

For example, consider the complex $\Delta$ containing a tetrahedron $\sigma$ as one of its facets. If $\partial\Delta\cap \sigma$ is just one edge, then clearly blowing-up along this edge produces a new facet of the boundary (see Figure \ref{tetra1}). Similarly, if $\partial\Delta\cap \sigma$ consists of two edges, then blowing-up produces two new facets in the boundary (see Figure \ref{tetra2}). However, if instead $\partial\Delta\cap \sigma$ consists of three edges containing a common vertex, then the blow-up along this center will produce \textit{four} new facets in the boundary of the blow-up -- one for each of the edges and one facet coming from the vertex joining the three edges (see Figure \ref{tetra4}). Such faces will be called simplicial. More precisely, a face $F\in\partial\Delta\cap \sigma$ of codimension $k$ is said to be \textbf{simplicial} if it is the intersection of $k$ codimension $k-1$ faces in $\partial\Delta\cap \sigma$ where $k>2$. 

\begin{figure}[h]
\centering
\begin{subfigure}{.45\textwidth}
\centering
  \includegraphics[width=0.5\linewidth]{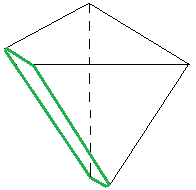}
 \caption{}
  \label{tetra1}
\end{subfigure}%
\begin{subfigure}{.5\textwidth}
  \centering
  \includegraphics[width=.45\linewidth]{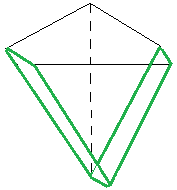}
  \caption{}
  \label{tetra2}
\end{subfigure}\\
\begin{subfigure}{.5\textwidth}
  \centering
  \includegraphics[width=.45\linewidth]{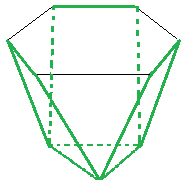}
  \caption{}
  \label{tetra4}
\end{subfigure}
\caption{}
\end{figure}

\begin{proposition}\label{New_Facets}
Let $\Delta$ be a Stanley-Reisner complex homeomorphic to a ball. Let $\tau_i$ and $\sigma_i$ be defined as above. Then there is a bijection between the set of nonempty maximal faces and simplicial faces of $\tau_i$ for $i=1,...,n$ and the set of new facets in $\partial\widetilde{\Delta}$.
\end{proposition}
\begin{proof}
Let $\sigma$ be a facet of $\Delta$ so that $\sigma = \sigma_i$ for some $i$. By Lemma 2.8, we may assume that $\sigma$ is the standard $n$-simplex since $\Delta$ is a Stanley-Reisner complex. Suppose that $F$ is a codimension $k>1$ face of $\sigma$ contained in $\tau=\tau_i$. By a change of coordinates we can assume that $F$ corresponds to $\{x_1=...=x_{k}=0\}\cap\sigma$ in $\mathbb{A}^n$. It is contained in the $k$ codimension $k-1$ faces 
\[F_1 =\{x_2=...=x_{k}=0\}\cap\sigma \text{ , ... , } F_k = \{x_1=...=x_{k-1}=0\}\cap\sigma.\]

\vspace{2mm}
First suppose that $F$ is a maximal face so that $F_j\notin \tau$ for any $j$. Then removing $F$ and taking the convex hull of the remaining lattice points leaves a region  bounded by $\{x_1+...+x_k=1\}$ which defines a new $(n-1)$-simplex in the boundary of $\widetilde{\Delta}$.

Next suppose that $F_1,...,F_k\in \tau$ so that $F$ is a simplicial face. For $k=2$, it is clear that the blow-up along $F_1\cup F_2$ does not produce an extra facet coming from the removal of $F$. Therefore let us consider $k>2$. Again removing $F_1$ and taking the convex hull of the remaining lattice points leaves a region bounded by $H_1: \{x_2+...+x_k=1\}$. In general removing $F_j$ leaves a region bounded by $H_j:\{x_1+...+x_n-x_j=1\}$. Therefore in the blow-up, when all $F_j$ are removed, the hyperplanes $H_j$ bound a new $(n-1)$-simplex in the boundary of $\widetilde{\Delta}$ defined by $\{x_1+...+x_k=1\}$, corresponding to the removal of the face $F$. It is clear that if any of the $H_j$ are not included in this calculation, then no new $n-1$-simplex is introduced by removing $F$, which completes the proof. 
\end{proof}

We would now like to produce a subword description for the faces identified in Proposition \ref{New_Facets}. Recall that by Proposition 1.12, a face $Q\setminus P$ is in the boundary of $\Delta(Q,\pi)$ if and only if $\dem(P)\neq\pi$. Furthermore, by Lemma 3.4 in \cite{KM}, $\dem(P) \geq \pi$ for every face $Q\setminus P$ in $\Delta(Q,\pi)$. Therefore faces $Q\setminus P$ are on the boundary when $\dem(P)>\pi$.

\begin{example}
\normalfont
Consider the Kazhdan-Lusztig variety, $X_{13425}\cap X^{34512}_{\circ}$. We can take $Q = (s_2,s_1,s_3,s_2,s_4,s_3)$ where $\prod Q$ is a reduced word for $34512$, and $s_2s_3$ is the only reduced word for $\pi =13425$. Then $\Delta(Q,13425)$ is the subword complex made with three tetrahedra, labeled in red:

\begin{figure}[h]
\centering
\includegraphics[width=0.5\linewidth]{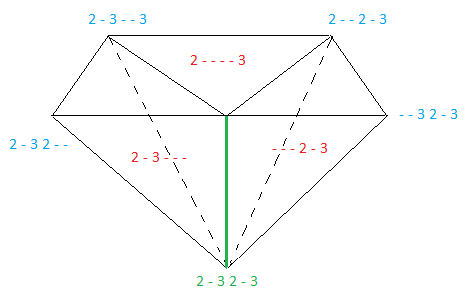}
\caption{}
\end{figure}

Any facet of $\partial\Delta$ produces a facet in $\partial\widetilde{\Delta}$ (albeit not a new facet), and corresponds to a word with Demazure product which covers $\pi$. Perhaps all facets in $\partial\widetilde{\Delta}$ come from faces of $\Delta$ corresponding to $P$ where $\dem(P)\gtrdot \pi$?

Notice that the edge $2-32-3$ (coloured in green) is contained in each of the three facets $2 - - - - 3$, $2-3---$ and $---2-3$. It is only in $2----3$ however that it corresponds to a new facet in $\partial\widetilde{\Delta}$ by Proposition \ref{New_Facets}. It has Demazure product $s_2s_3s_2$ which covers $s_2s_3$.  Indeed it is in this simplex that $2-32-3$ is maximal with respect to the property of $\dem(P)$ covering $\pi$. For example, the edge is contained in the interior faces $2 - 3 - - 3$ and $2 - - 2 - 3$ which have Demazure product equal to $\pi$.

It is worth noting that $2-32-3$ is not the only edge which has Demazure product which covers $\pi$. In fact  $\dem(213 - - 3)\gtrdot \pi$. However, it is not maximal because it is contained in the face $21 - - - 3$ which also covers $s_2s_3$. Finally, while we have maximality in terms of the faces, we have minimality in terms of the subwords themselves ($2-32-3$ contains the subword $2 - - - - 3$).
\end{example}

\begin{theorem}\label{max_face}
Let $\sigma_{Q\setminus P_1}$ be a facet of $\Delta(Q,\pi)$. Then $\tau_{Q\setminus P_2}$ is a maximal face of $\sigma_{Q\setminus P_1}\cap\partial\Delta(Q,\pi)$ if and only if $P_2$ is minimal among subwords $P\supset P_1$ satisfying $Dem(P)\gtrdot \pi$.
\end{theorem}
\begin{proof}

Suppose that $\tau_{Q\setminus P_2}$ is a maximal face of $\sigma_{Q\setminus P_1}\cap\partial\Delta(Q,\pi)$ as in Proposition \ref{New_Facets}. If $P_1$ has codimension $k$ in $P_2$, then there are simple reflections $\sigma_1,...,\sigma_k$ such that $\dem(P_2\setminus \sigma_i) =\pi$ (here $\tau_{Q\setminus P_2}$ is maximal so it is only contained in interior faces). We will proceed by induction on $k$. 

If $k=1$, so that $P_2$ is a facet of $\partial\Delta$, then $\dem(P_2)\gtrdot \pi$. Indeed we know that $P_2\setminus \sigma_1$ is a reduced expression for $\pi$, so inserting $\sigma_1$ increases the length by at most 1. If $\dem(P_2)$ did not cover $\pi$, then $\dem(P_2)=\pi$ and it would be in the interior of $\Delta(Q,w)$ which is a contradiction. 

It is worth demonstrating the $k=2$ case before proceeding with the general induction. Let $P_2 = T_1\sigma_1T_2\sigma_2T_3$ where $\dem(T_1T_2\sigma_2T_3)=\dem(T_1\sigma_1T_2T_3)=\pi$ and $T_1T_2T_3$ is a reduced word for $\pi$. Since $T_1$ is reduced and common to all of the words involved, it will appear at the beginning of any Demazure product of simple reflections. In other words, we may reduce to the case where $T_1$ is not present by noting that 

$$l(\dem(P_2))\leq l(\dem(\sigma_1T_2\sigma_2T_3)) + l(T_1).$$
Now let $T_2T_3$ be a reduced word for $\tau$. Then $\dem(T_2\sigma_2T_3)=\tau$ and 

$$l(\dem(\sigma_1T_2\sigma_2T_3))\leq 1+ l(\dem(T_2\sigma_2T_3)) = 1+ l(\tau).$$
Since $l(T_1)+ l(\tau) = l(\pi)$, we have 

$$l(\dem(P_2))\leq l(\pi) + 1$$
and since $\dem(P_2)$ cannot be $\pi$, we must have equality. 

For general $k$, we proceed in a similar way and write $ P_2 = T_1\sigma_1T_2\sigma_2T_3\dots\sigma_kT_{k+1}.$
By assumption, $\dem(P_2\setminus\sigma_i)=\pi$ and $T_1\cdot...\cdot T_{k+1}$ is a reduced expression for $\pi$. Then $T_{1}$ is a reduced word and 

$$ l(\dem(P_2)) \leq l(T_1)+ l(\sigma_1T_2\sigma_2T_3\dots\sigma_kT_{k+1}).$$
By induction $ l(\sigma_1T_2\sigma_2T_3\dots\sigma_kT_{k+1}) \leq 1+ l(T_2...T_{k+1})$, and the general result follows.

\end{proof}

\begin{corollary}\label{new_face}

Let $\sigma_{Q\setminus P_1}$ be a facet of $\Delta(Q,\pi)$. Then $\tau_{Q\setminus P_2}\subset \sigma_{Q\setminus P_1}$ corresponds to a new facet of $\partial\widetilde{\Delta(Q,\pi)}$ if $P_1$ has codimension $>1$ in $P_2$ and if $P_2$ is minimal among subwords $P\supset P_1$ such that $\dem(P)\gtrdot \pi$.

\end{corollary}
Let $P_1\supset P_2$ be two lists of simple reflections (as in the definition of a subword complex). We will say that $P_1$ has \textbf{codimension $k$} in $P_2$ if $P_1\setminus P_2$ consists of $k$ simple reflections (so that the face $Q\setminus P_1$ in $\Delta(Q,w)$ has codimension $k$ in $Q\setminus P_2$).

We know that the remaining list of faces from Proposition \ref{New_Facets} will come from intersections of maximal faces, and that these simplicial faces correspond to subwords containing the minimal ones described in Theorem \ref{max_face}. While this is not a convenient way to describe simplicial faces, our current description is enough to understand when a Kazhdan-Lusztig variety is not Gorenstein.

\begin{corollary}\label{KL_Gorenstein}
Consider the Kazhdan-Lusztig variety $X_w\cap X^v_{\circ}$ and let $Q$ be a reduced word for $v$ written as an ordered list of simple reflections.  Then $X_w\cap X^v_{\circ}$ is not Gorenstein if and only if there is a reduced subword $P_w\subset Q$ for $w$ and a codimension $>1$ subword $P\supset P_w$ that minimally satisfies $\dem(P)\gtrdot w$.
\end{corollary}
\begin{proof}
Since simplicial faces require that there exist maximal faces with codimension $>1$, it is necessary and sufficient to apply Corollary \ref{new_face} to the subword complex $\Delta(Q,w)$.
\end{proof}

\begin{example}\normalfont
Consider the Kazhdan-Lusztig variety $X_{31524}\cap X_{\circ}^{54321}$ which we know is not Gorenstein by \cite{Woo-Yong}. We will take $Q=(1,2,1,3,2,1,4,3,2,1)$ to be our reduced word for the permutation $54321$. Alternatively, we could show that $X_{31524}\cap X_{\circ}^{54321}$ is not Gorenstein using Corollary \ref{KL_Gorenstein}. One possible reduced word for $31524$ is $s_2s_1s_4s_3$, so we can choose $P_w = -2---143--$ as a reduced subword of $Q$ that represents the permutation $31524$ . This corresponds to the facet $Q\setminus P_w$ of $\Delta(Q,31524)$. 

We can check that it contains the codimension 1 interior faces $P_1=-21--143--$, $P_2=-2--2143--$ and $P_3=-2---143-1$. The pairwise intersection of these interior faces results in three codimension 2 faces of interest, namely 
\[P_{1,2}= -21-2143-- \text{ , } P_{1,3}= -21--143-1 \text{ and } P_{2,3}=-2--2143-1.\]
 The latter two are interior, while the face $P=-21-2143--$ has Demazure product which covers the permutation $31524$. By Corollary \ref{KL_Gorenstein}, $X_{31524}\cap X_{\circ}^{54321}$ is not Gorenstein.
\end{example}
\begin{example}
\normalfont
As another example, consider $X_{134526}\cap X_{\circ}^{345612}$. We will take $Q=(2,1,3,2,4,3,5,4)$. The only reduced word for $134526$ is $s_2s_3s_4$. There are three facets $P_1= 2-3-4---$, $P_2= 2-3----4$ and $P_3 =2----3-4$ which represent $s_2s_3s_4$. There is only one codimension 1 interior face for $P_1$, so there is no obstruction to being Gorenstein coming from this facet. On the other hand, $P_2$ is contained in $2-3-4--4$ and $2-3--3-4$, both interior, which in turn are both contained in $2-3-43-4$. Similarly, $P_3$ is contained in $2-3--3-4$ and $2---43-4$ (again both interior) which are contained in $2-3-43-4$. Each codimension 2 subword is minimal with respect to the property of $\dem(P)\gtrdot 134526$, so  by Corollary \ref{KL_Gorenstein}, $X_{31524}\cap X_{\circ}^{54321}$ is not Gorenstein.
\end{example}

It is worth noting that this subword criteria for checking the Gorenstein property is less cumbersome than checking the homology criteria from \cite{Stanley} (although the homology criteria is more general and works for Stanley-Reisner complexes which are not necessarily subword complexes).

\section{Degenerating blow-up algebras}\label{degen-commute}

To understand under what conditions the degeneration of the Rees algebra is the Rees algebra of the degeneration, we utilize some well-known results involving Gr\"obner bases and syzygies. Given a finite set of $\{p_1,...,p_r\}\subset S=k[x_1,...,x_k]$, we define

$$\Syz(\{p_i\}) = \Big\{ \sum a_i\epsilon_i\in \langle \epsilon_1,...,\epsilon_r\rangle \leq S[\epsilon_1,...,\epsilon_r]: \sum (a_i\epsilon_i)\big|_{\epsilon_i=p_i} = 0 \Big\}.$$

When dealing with syzygies, we usually require the coefficients $a_i\in S$, which is why we talk about syzygy modules over $S$, viewing the relations as the kernel of a map $S^r\rightarrow S$. In our case it is easier to deal with ideals, so we allow the coefficients $a_i$ to be in $S[\epsilon_1,...,\epsilon_r]$. Note that $\Syz(\{p_i\}) $ is just the ideal generated by the first syzygy equations (see Lemma \ref{classic-syzygy}).

The first lemma rewrites the Rees algebras we are concerned with in a convenient form to utilize results about syzygies.

\begin{lemma}\label{fg-algebra}
Let $I=\langle g_1,...,g_m\rangle\subset J=\langle g_{1},...,g_{m+n}\rangle$ be two ideals in $S=k[x_1,...,x_k]$. Denote $S/I$ and $J/I$ by $\bar{S}$ and $\bar{J}$ respectively. Then

$$\bar{S}[t\bar{J}] \cong S[\epsilon_{1},...,\epsilon_{m+n}]/(I+\langle\epsilon_1,...,\epsilon_m\rangle+\Syz(\{g_i\}))$$
as graded rings with the $\epsilon_i$ generating in degree 1.
\end{lemma}
\begin{proof}
Consider the map
$$\varphi: S[\epsilon_1,...,\epsilon_{m+n}]\rightarrow \bar{S}[t\bar{J}]$$
defined on generators by $\epsilon_i \rightarrow t\bar{g_i}$ (note that $tI\leq tJ$).  Let us compute $\ker(\varphi)$.

Given $f\in S[\epsilon_1,...,\epsilon_{m+n}]$, we can write $f= h_0 +...+h_k$ using the standard grading on $S[\epsilon_1,...,\epsilon_{m+n}]$, generated in degree 1 by the $\epsilon_i$. It then suffices to show the result for $f = f_{0}+\sum_{i=1}^{n}f_i\epsilon_i$ where $f_i\in S$ and extend to other degrees as needed. Now

$$\varphi(f) = \bar{f_{0}}+t\sum_{i=1}^{m+n}\bar{f_i}\bar{g_i}.$$

Then $\varphi(f)\in I$ iff $f_{0}\in I$ and $\sum_{i=m+1}^{n+m}f_ig_i \in I$ (the $g_i$ for $i=1,...,m$ are already in $I$).

Now $\sum_{i=m+1}^{m+n}f_ig_i\in I$ means that there exists $h_i\in S$ such that $\sum_{i=m+1}^{m+n}f_ig_i - \sum_{j=1}^{m}h_jg_j = 0$ in $S$. This relation is contained in $\Syz(\{g_i\})$. Then
$$f = f_{0}+\sum_{i=m+1}^{m+n}f_i\epsilon_i  - \sum_{j=1}^{m}h_j\epsilon_j + \sum_{j=1}^{m}(f_j+h_j)\epsilon_j \in I+\langle\epsilon_1,...,\epsilon_m\rangle+\Syz(\{g_i\}).$$

The reverse containment is easy to check.
\end{proof}

Writing the Rees algebra in this way allows us to work in $S$ instead of $S/I$ and to then use standard results about computing syzygies using Gr\"obner bases. To properly rewrite the Rees algebra in this form however, we needed to consider the syzygies on the generators of both $I$ and $J$ instead of just $J$ (which gives the extra relations on generators of $J$ that land in $I$). The phrasing of the next lemma in using both $I$ and $J$ would otherwise seem a little strange without viewing it in this context.

\begin{lemma}\label{classic-syzygy}
Using the notation of Lemma \ref{fg-algebra}, suppose that the generators for $I$ and $J$ are also Gr\"obner bases with respect to some monomial ordering on $S$. Define

$$\sigma_{ij} = m_{ji}\epsilon_i - m_{ij}\epsilon_j$$
$$\tau_{ij} = m_{ji}\epsilon_i - m_{ij}\epsilon_j - \sum_u f_u^{(ij)}\epsilon_u$$
where $m_{ij} = \init(g_i)/GCD(\init(g_i),\init(g_j))\in S$ and the $f_u^{(ij)}$ come from the division algorithm applied to $m_{ji}g_i - m_{ij}g_j$ . Then $\Syz(\{g_i\})$ is generated by the $\tau_{ij}$ and $\Syz(\{\init(g_i)\})$ is generated by the $\sigma_{ij}$.
\end{lemma}
\begin{proof}
Note that since $I\subset J, \{g_1,...,g_m,g_{m+1},...,g_{m+n}\}$ is a Gr\"obner basis for $J = I+J$. The result follows from Lemma 15.1, Theorem 15.8 (Buchberger's Criterion), and Theorem 15.10 (Schreyer's theorem) in \cite{Eisenbud}.
\end{proof}

We now define the crucial ordering on the extra variables $\epsilon_i$  to introduce a meaningful degeneration of the blow-up algebra.

\begin{lemma}\label{syzygy-weight}
Let $>_{\lambda}$ be a term order on $S=k[x_1,...,x_k]$ defined by an integral weight function $\lambda: \mathbb{Z}^k\rightarrow \mathbb{Z}$. Using the notation of Lemma \ref{classic-syzygy}, consider the integral weighting $\tilde{\lambda}$ defined on generators by

$$
\tilde{\lambda}(z) = \left\{
        \begin{array}{ll}
            \lambda(\init_{>_{\lambda}}(g_i)) & \text{if } z=\epsilon_i,\\
            \lambda(x_j) & \text{if } z=x_j
        \end{array}
    \right.
$$

Then, $\{g_1,...,g_m,\epsilon_1,...,\epsilon_m\}\cup\{\tau_{ij}\}$ is a Gr\"obner basis with respect to $<_{\tilde{\lambda}}$ for $I+\langle\epsilon_1,...,\epsilon_m\rangle+\Syz(\{g_i\})$ and its initial ideal is
$$\init_{>_{\lambda}}(I)+ \langle\epsilon_1,...,\epsilon_m\rangle+\Syz(\{\init_{>_{\lambda}}(g_i)\}).$$
\end{lemma}

\begin{proof}
This is a variation of the proof in Schreyer's theorem (see Theorem 15.10 in \cite{Eisenbud}). Observe that

$$\init(\tau_{ij}) = m_{ji}\epsilon_i - m_{ij}\epsilon_j$$
since $m_{ji}\init(g_i) = m_{ij}\init(g_j)$ and because these terms are greater than any that appear in the $f_u^{(ij)}\epsilon_u$ as a result of the division algorithm.

First, we will show that the $\tau_{ij}$ don't just generate $\Syz(\{g_i\})$ but are a Gr\"obner basis for it. Let $\tau = \sum h_v\epsilon_v$ be a syzygy of $\{g_i\}$ (with terms that can be cancelled already eliminated). For each index $v$, set $n_v\epsilon_v = \init(h_v\epsilon_v)$. Let $q$ be the maximum weight of a term in $\tau$. Then let $\sigma = \sum n_v\epsilon_v$ such that $\tilde{\lambda}(n_v\epsilon_v) =q$. That is, $\init(\tau)=\sigma$.

Notice that since $\tau$ was a syzygy of the  $g_i$, $\sigma$ must be a syzygy of the $\init(g_i)$. Indeed $\tau$ gives rise to the relation

$$\sum_{\text{weight $q$}}n_v\epsilon_v + \sum_{u}a_u\epsilon_u =0$$
by separating the terms from the relation provided by $\sigma$ with the remaining terms. Then

$$\sum_{\text{weight $q$}}n_v\epsilon_v = -\sum_{u}a_u\epsilon_u.$$

If $\sum_{\text{weight $q$}}n_v\epsilon_v \neq 0$, then the left hand side would contain a term with weight $q$. The right hand side necessarily has weight strictly less than $q$, a contradiction.

Thus $\sigma$ is a syzygy of the $\init(g_i)$ and therefore is in the ideal generated by the $\sigma_{ij}$ by Lemma \ref{classic-syzygy}. Since $\init(\tau_{ij}) = \sigma_{ij}$, we are done.\\

Next, let $f\in I+\langle\epsilon_1,...,\epsilon_m\rangle+\Syz(\{g_i\})$. The worry at this point should be that cancellation results in an $\init(f)$ not in our target ideal. For example, perhaps some element in $I+\langle\epsilon_1,...,\epsilon_m\rangle$ can be used to eliminate the $m_{ji}\epsilon_i - m_{ij}\epsilon_j$ in some $\tau_{ij}$ leaving only $\sum_u f_u^{(ij)}\epsilon_u$ terms not found in $\Syz(\{\init (g_i)\})$.

Let $m$ be a term in $\init(f)$. Clearly if $m$ was a leading term for an element in $I+ \langle\epsilon_1,...,\epsilon_m\rangle$, then $m\in \init(I)+ \langle\epsilon_1,...,\epsilon_m\rangle$. Otherwise, $m$ is a leading term of some combination involving a $\tau_{ij}$, and there are 2 cases to consider. Fix $(i,j) = (i_0,j_0)$.

If $m$ is a multiple of a term in $\tau_{i_0j_0}$ where at least one $\epsilon_{i_0}$ and $\epsilon_{j_0}$ are in $\langle\epsilon_{m+1},...,\epsilon_{m+n}\rangle$, we are done since no cancellation is possible to eliminate these leading terms, so $m\in \init(I)+ \langle\epsilon_1,...,\epsilon_m\rangle+\Syz(\{\init(g_i)\})$

If both $\epsilon_{i_0}$ and $\epsilon_{j_0}$ are in $\langle\epsilon_1,...,\epsilon_m\rangle$, then the division algorithm applied to $(m_{j_0i_0}\epsilon_{i_0} - m_{i_0j_0}\epsilon_{j_0})|_{\epsilon_i=g_i} \in I$ would have produced $f_u^{(i_0j_0)}$ also contained in $I$ since the $g_1,...,g_m$ are a Gr\"obner basis for $I$, and such a decomposition for ideal containment is unique. Then $m\in \init(I)+ \langle\epsilon_1,...,\epsilon_m\rangle$, completing the proof.

\end{proof}

\begin{theorem}\label{commute-weight}
Let $>_{\lambda}$ be a term order on $S=k[x_1,...,x_k]$ defined by an integral weight function $\lambda: \mathbb{Z}^k\rightarrow \mathbb{Z}$. Let $I=\langle g_1,...,g_m\rangle\subset J=\langle g_{1},...,g_{m+n}\rangle$ be two ideals in $S$ each generated by a Gr\"obner basis (with respect to $\lambda$). Then there exists an integral weight function $\tilde{\lambda}:\mathbb{Z}^{n+m+k}\rightarrow \mathbb{Z}$ defining a term order $>_{\tilde{\lambda}}$ on the blow-up algebra of $S/I$ along $J/I$ which Gr\"obner degenerates to the blow-up algebra of $ S/\init_{>_{\lambda}}(I)$ along $\init_{>_{\lambda}}(J)/\init_{>_{\lambda}}(I)$.

\end{theorem}
\begin{proof}
By Lemma \ref{fg-algebra}, the blow-up algebra $\bar{S}[t\bar{J}]$ is isomorphic to

$$S[\epsilon_{1},...,\epsilon_{m+n}]/(I+\langle\epsilon_1,...,\epsilon_m\rangle+\Syz(\{g_i\}))$$
We know that this degenerates to

$$S[\epsilon_{1},...,\epsilon_{m+n}]/(\init_{>_{\lambda}}(I)+\langle\epsilon_1,...,\epsilon_m\rangle+\Syz(\{\init_{>_{\lambda}}(g_i)\}))$$
using the integral weighting $\tilde{\lambda}$ defined in Lemma \ref{syzygy-weight}.

On the other hand, $S/I$ and $J/I$ degenerate to $S/\init_{>_{\lambda}}(I)$ and $\init_{>_{\lambda}}(J)/\init_{>_{\lambda}}(I)$. By Lemma \ref{fg-algebra}, the blow-up algebra of $S/\init_{>_{\lambda}}(I)$ along $\init_{>_{\lambda}}(J)/\init_{>_{\lambda}}(I)$ is

$$S[\epsilon_{1},...,\epsilon_{m+n}]/(\init_{>_{\lambda}}(I)+\langle\epsilon_1,...,\epsilon_m\rangle+\Syz(\{\init_{>_{\lambda}}(g_i)\}))$$
as required.

\end{proof}

\begin{corollary}\label{commute}
The blow-up of a Kazhdan-Lusztig variety along its boundary commutes with the degeneration in Section \ref{toric} using the integral weight function $\tilde{\lambda}$ in Theorem \ref{commute-weight}.
\end{corollary}
\begin{proof}
Let $I$ be the generating ideal for the Kazhdan-Lusztig variety and $J$ the generating ideal for its boundary. Choose an integral weight function $\lambda$ that produces the lexicographic ordering for the equations in $I$ and $J$ (ie. the $\lambda$ needed for the degeneration in Section \ref{toric}). The result follows.
\end{proof}

Finally, the fact that we are blowing-up an entire divisor results in the possibility of the exceptional locus not being a divisor, which is usually not an issue when blowing-up higher codimension centers (as found in desingularization algorithms for example). Luckily the exceptional locus in our case actually is a divisor.

\begin{corollary}
The exceptional locus in the blow-up of a Kazhdan-Lusztig variety along its boundary is a divisor.
\end{corollary}
\begin{proof}
The blow-up of a Kazhdan-Lusztig variety along its boundary degenerates to the blow-up of a Stanley-Reisner scheme along its boundary, and the exceptional locus of the former degenerates to the exceptional locus of the latter. By Section \ref{SR_Exceptional}, the exceptional locus in the Stanley-Reisner case is a divisor, so the result follows by upper semi-continuity of dimension.
\end{proof}

\section{The blow-up of $X^w$ along $\partial X^w$}\label{semicontinuous}

We need to show that the general definition of Gorenstein presented in \ref{anticanonical_reducible} is an uppersemicontinuous property. This means that it is an open condition in flat families so that checking the Gorenstein property on the special fibre is sufficient for showing it on the general fibre. To achieve this, we need to check that the property of a sheaf being dualizing is open.

Let $X$ be a proper scheme of dimension $n$ over $k$. Recall that $\kappa$ is said to be a dualizing sheaf if:

\begin{itemize}
\item $H^n(X,\kappa) \cong k$\\
\item For all coherent sheaves $\mathcal{F}$, the following map is a perfect pairing:

$$H^i(X,\mathcal{F})\times H^{n-i}(X,\mathcal{F}^*\otimes \kappa)\rightarrow H^n(X,\kappa). $$
\end{itemize}

Suppose we are given a flat family $M$ over $\mathbb{A}^1$ (where $M$ is a proper scheme of dimension $n$ over $k$) such that $M_t$ Gr\"obner degenerates to $M_0$. Suppose we also have a coherent sheaf $\kappa_t$ over $M_t$ that degenerates to $\kappa_0$ over $M_0$. Then the pair $(M_t,\kappa_t)$ has been Gr\"obner degenerated to $(M_0,\kappa_0)$. We would like to know whether $\kappa_0$ being a dualizing sheaf implies that $\kappa_t$ is dualizing as well. 

The perfect pairing property from the definition is already an open condition since it can be rephrased as a map into the dual space being an isomorphism. The first condition however may not be open in general. In fact $H^n(M_t,\kappa_t)$ can jump in dimension on the special fibre. In particular, if $H^n(M_0,\kappa_0)\cong k$, then $H^n(M_t,\kappa_t)\cong k \text{ or } 0$. What is true however, is that 
$$H^{n-1}(M_0,\kappa_0)= 0\Rightarrow H^n(M_t,\kappa_t)\cong k  $$
Indeed, $H^{n-1}(M_0,\kappa_0)=0$ implies that no such jump in dimension occurred at the special fibre, therefore $H^n(M_t,\kappa_t)\cong k$ as required.

\begin{proposition}
Let $(M_t,\kappa_t)$ Gr\"obner degenerate to $(M_0,\kappa_0)$ as above. Suppose also that $H^{n-1}(M_0,\kappa_0)=0$. Then $\kappa_t$ is a dualizing sheaf if $\kappa_0$ is.
\end{proposition}

In the case that concerns us, we know that $\widetilde{X^w\cap X_v^{\circ}}$ Gr\"obner degenerates to $M_0 = \proj(R[t(\init J_{w,v})])$ (the total transform of a Stanley-Reisner scheme), extending the degeneration from Section \ref{toric}. In particular, the exceptional divisor $E$ in  $\widetilde{X^w\cap X_v^{\circ}}$ must degenerate to the boundary divisor $E_0=\partial M_0$ (which was the exceptional divisor in the toric case).

We know that $E_0$ is a Cartier anticanonical divisor (as defined in \ref{anticanonical_reducible}) which we can use to define our prospective $\kappa_0$. For each $M_t$ in the degeneration above, we have a reduced Cartier exceptional divisor $E_t$ such that $-E_t$ is also Cartier. We can then define $\kappa_t = \mathcal{O}(-E_t)$. Since $\kappa_t$ is coherent, we can give it a finite presentation, and take the Gr\"obner limit in the Quot scheme to define $\kappa_0$.

We need only check that it is dualizing and that $H^{n-1}(M_0,\kappa_0)=0$. In fact we can deduce these cohomology conditions on $M_0$ from its normalization. 

We know that $M_0$ is equidimensional and has normal components. We will first show that $\kappa_0$ is a dualizing sheaf when $M_0$ has two components whose intersection has codimension one and is also normal. Say $M_0 = X\cup Y$. Given a coherent sheaf $\mathcal{F}$ on $M_0$, we have the short exact sequence:

$$0\rightarrow \mathcal{F} \rightarrow \mathcal{F}|_X\oplus\mathcal{F}|_Y \rightarrow \mathcal{F}|_{X\cap Y} \rightarrow 0$$
which gives rise to the long exact sequence in sheaf cohomology:

$$...\rightarrow H^{i-1}(\mathcal{F}|_{X\cap Y})\rightarrow H^{i}(\mathcal{F}|_{X\cup Y})\rightarrow H^{i}(\mathcal{F}|_{X\sqcup Y})\rightarrow ....   $$

Let $\mathcal{F}=\kappa_0$. Since the normalization of $M_0$ is just $X\sqcup Y$ (which is normal), and since $X\cap Y$ has dimension $n-1$ and is also normal, we have that 
$$0=H^{n-2}(X\cap Y,\mathcal{F}|_{X\cap Y}) \rightarrow H^{n-1}(X\cup Y,\mathcal{F}|_{X\cup Y}) \rightarrow H^{n-1}(X\sqcup Y,\mathcal{F}|_{X\sqcup Y}) =0$$ 
forcing $ H^{n-1}(X\cup Y,\mathcal{F}|_{X\cup Y})=0$ as required.

The perfect pairing condition also follows from the normal cases $X\sqcup Y$ and $X\cap Y$. For general coherent sheaves $\mathcal{F}$, we have a natural map 
$$\pi_Z: H^{i}(Z,\mathcal{F}|_{Z}) \rightarrow H^{n-i}(Z,(\mathcal{F}^*\otimes \kappa_0)|_{Z})^*$$ which is an isomorphism for $Z = X\sqcup Y, X\cap Y$. Let $\mathcal{G} = \mathcal{F}^*\otimes \kappa_0.$ We can draw the commuting diagram:

\[\begin{tikzcd}[column sep=0.75em]
H^{i-1}(\mathcal{F}|_{X\sqcup Y}) \arrow{r}\arrow[d] & H^{i}(\mathcal{F}|_{X\cap Y}) \arrow[r]\arrow[d]& H^{i}(\mathcal{F}|_{X\cup Y}) \arrow[r]\arrow[d] & H^{i}(\mathcal{F}|_{X\sqcup Y}) \arrow[r]\arrow[d] & H^{i+1}(\mathcal{F}|_{X\cap Y}) \arrow[d]  \\
H^{n-i+1}(\mathcal{G}|_{X\sqcup Y})^* \arrow[r] & H^{n-i}(\mathcal{G}|_{X\cap Y})^* \arrow[r] & H^{n-i}(\mathcal{G}|_{X\cup Y})^* \arrow[r] & H^{n-i}(\mathcal{G}|_{X\sqcup Y})^* \arrow[r] & H^{n-i-1}(\mathcal{G}|_{X\cap Y})^* \\
\end{tikzcd}\]
By the five lemma, $\pi_{X\cup Y}$ will also be an isomorphism.

We can now prove the result for a general union using the inductive version of the degeneration of a Kazhdan-Lusztig variety $KL_1$ to the Stanley-Reisner scheme (see Section \ref{GVD}). Recall that $KL_1$ can be degenerated as

$$KL_1 \rightsquigarrow KL_2 \cup_{0 \times KL_3} (\mathbb{A}^1 \times KL_3)$$
as shown in \cite{Knut-patches}. 

We have proven the base case above. If we now assume that the result holds for $X=KL_2, Y=KL_3$, then the same argument shows that the result holds for $KL_1$ too.

In the blow-up along the boundary divisor, the intersection of components (corresponding to interior facets of the subword complex) satisfy the same conditions as the Stanley-Reisner case. Therefore the sheaf associated to the Cartier boundary divisor of the total transform is dualizing.


We are now ready to state our main result:

\begin{theorem}
Let $I_{w,v}\leq k[z_{i,j}]$ be the defining ideal for Kazhdan-Lusztig variety $X^w\cap X^{\circ}_v$. The boundary divisor is defined by $J_{w,v} = \displaystyle\cap_{u\lessdot w} I_{u,v}$. Then the blow-up of $S = k[z_{ij}]/ I_{w,v}$ along $J_{w,v}$ is Gorenstein.
\end{theorem}
\begin{proof}
By Theorem \ref{Knutson09} and Corollary \ref{Knutson-corollary}, there exists a degeneration of the pair $(S , J_{w,v})$ to $(R, \init (J_{w,v}))$ where $R = k[z_{ij}]/ \init (I_{w,v})$. Here $R$ is the coordinate ring for a Stanley-Reisner scheme. In fact $\init (I_{w,v})$ is the Stanley-Reisner ideal of a subword complex $\Delta(Q,w)$. Viewing $\spec(R)$ as a toric scheme, we can associate to it an affine semigroup in the character lattice which can be viewed as the cone on the polytope $\Delta(Q,w)$ (in appropriate coordinates). Here $\init (J_{w,v})$ corresponds to the boundary of this convex region.

By Corollary \ref{BU-gorenstein}, the blow-up of $R$ along $\init (J_{w,v})$ is Gorenstein. Furthermore, by Corollary \ref{commute}, there exists a degeneration of the blow-up algebra $S[tJ_{w,v}]$ to $R[t(\init J_{w,v})]$. We have the following diagram:

\[\displaystyle \large\begin{tikzcd}
S[tJ_{w,v}] \arrow[r, rightsquigarrow]
& R[t(\init J_{w,v})]  \\
(S , J_{w,v})\arrow[r, rightsquigarrow, "lexinit"]\arrow[u]
& (R, \init J_{w,v})\arrow[u]
\end{tikzcd}\]

Finally, we need to determine whether being Gorenstein (in the more general sense mentioned in Section \ref{toric}) is open in flat families. The beginning of Section \ref{semicontinuous} shows the semicontinuity of the Gorenstein property.


\end{proof}

\begin{corollary}
The blow-up of $X^w$ along $\partial X^w$ is Gorenstein.
\end{corollary}
\begin{proof}
See Proposition \ref{Gorenstein-local}.
\end{proof}

\begin{example}\label{BU-KL}\normalfont
Consider $X^{53241}\cap X_{12345}^{\circ}$ with
$$I_{w,v} = \langle z_{11}z_{22} - z_{12}z_{21}, z_{11}z_{23} - z_{13}z_{21}, z_{12}z_{23} - z_{13}z_{22}\rangle$$
as in Example \ref{degen-example}, defined by vanishing of all $2\times 2$ minors of
 \[ \left( \begin{array}{ccc}
 z_{21} & z_{22} & z_{23}  \\
 z_{11} & z_{12} & z_{13} \end{array} \right)\]

 The blow-up along the boundary satisfies the following equations (with $\epsilon_1,\epsilon_2$ in degree one):
 $$I = \langle z_{11}z_{22} - z_{12}z_{21}, z_{11}z_{23} - z_{13}z_{21}, z_{12}z_{23} - z_{13}z_{22},$$
 $$\epsilon_1z_{11} - \epsilon_2z_{21},\epsilon_1z_{12} - \epsilon_2z_{22},\epsilon_1z_{13} - \epsilon_2z_{23} \rangle$$
Those are the $2\times 2$ minors of

 \[ \left( \begin{array}{cccc}
 z_{21} & z_{22} & z_{23} & \epsilon_1 \\
 z_{11} & z_{12} & z_{13} & \epsilon_2\end{array} \right).\]

These equations look like the ones defining $X^{643251}\cap X_{123456}^{\circ}$. A quick check of the Gorenstein criteria for Schubert varieties shows that $X^{643251}$ is not Gorenstein, and indeed we need to account for $\epsilon_1$ and $\epsilon_2$ being in degree one. Therefore we take the GIT quotient $X^{643251}\cap X_{123456}^{\circ}$ $ //\mathbb{C}^{\times}$ where $\mathbb{C}^{\times}$ acts diagonally on $(\epsilon_1,\epsilon_2)$.

Now we should check that the non-Gorenstein locus is contained in the subvariety where $\epsilon_1=\epsilon_2=0$ (the subset quotiented out). For example $X^{643251}$ is not Gorenstein at $e_{432165}$, but can be defined by the vanishing of all of the coordinates in
\[ \left( \begin{array}{cccc}
 z_{21} & z_{22} & z_{23} & \epsilon_1 \\
 z_{11} & z_{12} & z_{13} & \epsilon_2\end{array} \right)\]
 and hence is not present in $X^{643251}\cap X_{123456}^{\circ}$ $ //\mathbb{C}^{\times}$.
 \qed
\end{example}

The other three $n=5$ examples in type A are also of this form. While we don't expect $\widetilde{X^w}$ to always be a GIT quotient of a Kazhdan-Lusztig variety, this example serves as a good transition to the next important question - how can we best interpret $\widetilde{X^w}$? Is it isomorphic to some well-known variety? Can we determine the exceptional locus of the blow-up map? Does $\widetilde{X^w}$ enjoy any other useful properties?

The first approach is to study $\widetilde{X^w}$ locally, through local equations for the blow-up algebra using syzygies. This is the content found in the remainder of this section. The second approach is a more global approach which deduces information about $\widetilde{X^w}$ using Bott-Samelson maps. This is the content of Section \ref{understanding-total-transform}.

Focusing on the local question for now, we recall that the boundary of a Kazhdan-Lusztig variety $X^w\cap X_v^{\circ}$ is defined by the ideal

$$ \bigcap_{u\lessdot w} I_{u,v}  $$

By the work of Bertiger in \cite{Bertiger}, we know that $\cap I_{u,v}$ is generated by certain products of determinants. Understanding the equations for the blow-up is the same as understanding the syzygies of this intersection.  What are the syzygies of $\cap I_{u,v}$? Are they determinantal?

We know that the equations in $I_{u,v}$ can be written as the sum of similar ideals where $u$ is \textbf{bigrassmanian}. Recall that $\pi$ is bigrassmannian if $\pi$ and $\pi^{-1}$ admit a unique descent, and this is the same as $X^{\pi}$ being defined using one rank condition. Can we solve this simpler problem of understanding the syzygies of $I_{u,v}$ if $u$ is bigrassmanian?

\subsection{Syzygies of products of ideals}
Let $I_1,...,I_r\subset R$ be ideals of a Noetherian integral domain $R$. Suppose that $I_k = \langle g_{k,1},...,g_{k,n_k}\rangle$ and consider the ring $S = R[z_{1,1},...,z_{r,n_r}]/\langle z_{i,j} - z_{k,l}\rangle $ with one $z_{i,j}$ for each generator of each ideal $I_k$ and with equal generators identified. We will denote the ideal of syzygies of $\{g_{k,i}\}_{i=1}^{n_k}$ by $\Syz(I_k)$. That is, $\Syz(I_k)\subset S$ will consist of all expressions of the form

$$ \sum_{i=1}^{n_k}a_iz_{k,i} \text{, where } a_i\in S \text{ such that }\sum_{i=1}^{n_k}(a_iz_{k,i})\displaystyle\big|_{z_{i,j} = g_{i,j}} =0$$

Similarly,
$$\Syz(\prod I_k)\subset R[\{z_{1,i_1}\cdot ... \cdot z_{r,i_r}\}]/\langle z_{i,j} - z_{k,l}\rangle \subset S$$
consists of the syzygies of the $N=\prod_{i=1}^r n_i$ products of the form $g_{1,1},...,g_{r,n_r}$.

We would like to understand $\Syz(\prod I_k)$ in terms of the $\Syz(I_k)$. Let $I = \{(i_1,...,i_r), 1\leq i_j\leq n_j\}$. An element $0\neq z \in \Syz(\prod I_k)$ has the form

$$ \sum_{(i_1,...,i_r)\in I}a_{(i_1,...,i_r)}z_{1,i_1}\cdot ...\cdot z_{r,i_r}$$

By viewing this as an equation written in the $\{z_{1,j}\}_{j=1}^{n_1}$, we see that either all the $z_{1,j}$ are equal so that we can factor them out, or at least 2 are distinct and $z\in \Syz(I_1)$. By next looking at the $\{z_{2,j}\}_{j=1}^{n_2}$, we see that either all of the $z_{2,j}$ are equal so that we can factor them out or $z\in \Syz(I_2)$.

Continuing in this way, we see that $z$ is contained in the intersection of some $\Syz(I_k)$ (it must be contained in at least one of these, otherwise all the $z_{1,i_1}\cdot ... \cdot z_{r,i_r}$ were equal and $z=0$).

Let us denote the ideal $\langle z_{k,1},...,z_{k,n_k}\rangle$ by $\bar{I_k}$. The above argument shows that

$$z\in \Syz(I_1)\cap ...\cap\Syz(I_r) + \bar{I_1}(\Syz(I_2)\cap ... \cap \Syz(I_r)) +  ... $$
$$+ \bar{I_r}(\Syz(I_1)\cap ... \cap \Syz(I_{r-1}))+...+ \bar{I_1}\cdot...\cdot \overline{I_{r-1}}\Syz(I_r)$$

Put in a more compact form, we have shown that

$$\Syz(\prod_{k=1}^rI_k) \subset\displaystyle \sum_{\emptyset\neq M\subset\{1,...,r\}} (\prod_{k\in M^c}\bar{I_k})(\bigcap_{k\in M} \Syz(I_k)) $$

For the reverse containment, observe that for $z\in \Syz(I_k)\cap\Syz(I_j)$ we can write $z$ in two ways

$$  a_1z_{k,1} + ... + a_{n_k}z_{k,n_k} = b_1z_{j,1} +...+b_{n_j}z_{j,n_j}  $$

We will assume that any cancellation on each side has already taken place. This means that each $a_m\in \bar{I_j}$ as a result. In particular both sides contain products of the form $z_{k,i_k}z_{j,i_j}$. This shows that $z\in \Syz(I_jI_k)$. Generalizing this we get the reverse containment, as required.

\begin{theorem}
$\Syz(\prod_{k=1}^rI_k) = \displaystyle \sum_{\emptyset\neq M\subset\{1,...,r\}} (\prod_{k\in M^c}\bar{I_k})(\bigcap_{k\in M} \Syz(I_k)).$
\end{theorem}

\subsection{Viewing syzygies of intersections in terms of products}

Building on the ideas from the previous section, our goal is to ultimately understand the syzygies of $\cap I_k$ in terms of the syzygies of $\prod I_k$ which we studied in the last section.

Let $T = R[w_1,...,w_n]$ (simplified from the $z_{i,j}$ coordinates used in the previous section) and observe that
$$\Syz(\Pi I_k)\text{ , }\Syz(\cap I_k)\subset T$$
where $w_i$ is associated to $g_i\in R$. 

Consider $z \in \Syz(\cap I_k)$. Then
$$z = a_1y_1+...+a_py_p$$
where $a_i\in R$ and $y_i$ corresponds to a generator of $\cap I_k$ (using \cite{Bertiger} to fix generators). Then $z^r$ involves products of $r$ choices of $y_i$ which when restricted to $g_j$ must be in $\prod I_k$ (indeed $\cap I_k\subset \sum I_k$). That is to say that $z^r \in \Syz(\prod I_k)$. Finally, since $\prod I_k\subset \cap I_k$, we have the same containment regarding syzygies.

We have thus shown

$$Syz(\cap I_k)\subset \sqrt{\Syz(\Pi I_k)} \subset \sqrt{\Syz(\cap I_k)}$$

The next lemma provides the last piece to proving the main result for this section.

\begin{lemma}
If $\cap I_k $ is reduced, then so is $\Syz(\cap I_k)$.
\end{lemma}
\begin{proof}
Given an element $z =a_0+a_1w_1+...+a_nw_n$ of $R[w_1,...,w_n]$, suppose that

$$ (a_0+a_1w_1+...+a_nw_n)^m \in \Syz(\cap I_k)   $$

Then $a_0^m$ and $(a_iw_i)^m$ must be in $\cap I_k$ after evaluating at $w_i=g_i$. Since $\cap I_k$ is reduced, the restriction of each $a_iw_i$ must also be in $\cap I_k$. Therefore $z\in \Syz(\cap I_k)$.
\end{proof}

Using Frobenius splittings to show this result might also prove useful. Finally, we have:

\begin{proposition}
$\sqrt{Syz(\prod I_w)} = Syz(\cap I_w)$.
\end{proposition}

This shows that understanding the syzygies of the product is the same as understanding the syzygies of the intersection.

\subsection{Specializing to bigrassmannians}
We know that a Schubert variety $X^w$ can be decomposed as

$$X^w = \displaystyle\bigcap _{\substack{\nu\geq w, \\ \text{ $w_0\nu$
bigrassmannian}}} X^{\nu}.$$
Bigrassmannians are especially easy to work with since they involve only one rank condition.

In our problem, we are blowing-up along the ideal $\cap J^{\sigma}$ where each $J^{\sigma}$ is the defining ideal for a component of the boundary divisor of $X^w\cap X_{\tau}^{\circ}$ (which is defined by $I$). Since the boundary has dimension 1 less than $X^w\cap X_{\tau}^{\circ}$, we can write

$$ J^{\sigma} = I + I^{\nu}$$

where $\nu$ is a bigrassmannian.

Since a bigrassmannian $\nu$ is defined by the vanishing of all $k\times k$ minors in an $m\times n$ box, we can define a simple bigrassmannian $\nu_s$ as the case when $m=k$ and $n=k+1$. Such simple bigrassmannians have a special property.

\begin{lemma}
The syzygy ideal of $I^{\nu_s}$, with $\nu_s$ a simple bigrassmannian, is a determinantal variety.
\end{lemma}

In fact the equations for $I^{\nu_s}$ can be found in the same was as Example \ref{BU-KL} by adding an extra row and column to the matrix for $\nu_s$. This corresponds to the defining equations for a Kazdhan-Lusztig variety. When $\nu$ is a bigrassmannian, we know exactly what the generators are for the syzygy ideal of $I^{\nu}$ thanks to \cite{Ma}.

To make use of these facts, we notice that since
$$ \displaystyle \sum_{\emptyset\neq M\subset\{1,...,r\}} (\prod_{k\in M^c}\bar{J^{\sigma_k}})(\bigcap_{k\in M} \Syz(J^{\sigma_k})) =\Syz(\prod J^{\sigma})\subset\Syz(\cap J^{\sigma}) $$
there are a number of ideals for Kazhdan-Lusztig varieties corresponding to simple bigrassmannians contained in this sum of ideals, some of which might be interesting to consider. In particular, with some work we could make some mild conjectures: perhaps the blow-up of a Kazhdan-Lusztig variety along its boundary is contained in some product of Kazhdan-Lusztig varieties (with some GIT quotient to account for certain variables in degree one). This would generalize the results we observed in low dimensions. However the containment at this time doesn't seem to be explicit enough to get a useful form for the total transform. Extending the $B$-action to the total transform might identify it as a $B$-invariant subvariety of this product.

The fact that a product of Kazhdan-Lusztig varieties is appearing in our local computations suggests that perhaps globally we can describe the blowup as some product of Schubert-like varieties. This is where generalized Bott-Samelson varieties come into the picture, further described in the next section.

\section{Understanding the total transform}\label{understanding-total-transform}

To get an explicit modular interpretation of blow-up of $X^w$ along $\partial X^w$, we seek a way to compare $\widetilde{X^w}$ to other well-studied varieties.  We will exploit Bott-Samelson resolutions of $X^w$ to get convenient maps from Bott-Samelson varieties onto $\widetilde{X^w}$. To obtain this map, we first recall the universal property of blow-up maps.

 \subsubsection{Universal property of blow-ups} The blow-up of a scheme $X$ along some closed subscheme $Y$ is a fibre diagram:

  \[
  \begin{tikzcd}
  E \arrow[swap]{d}\arrow{r} & Bl_Y(X)\arrow{d}\\
  Y \arrow{r} & X
  \end{tikzcd}
  \]
 where $E$ is an effective Cartier divisor, such that any other map to $X$ with an effective Cartier divisor as the fibre over $Y$ factors through the diagram. The only Cartier divisors considered in the sections that follow are effective.

\begin{lemma}
Let $\pi:\widetilde{X^w}\rightarrow X^w$ be the blow-up of $X^w$ along $\partial X^w$. Suppose $\varphi_Q:BS^Q\rightarrow X^w$ is a generalized Bott-Samelson resolution such that $\varphi_Q^{-1}(\partial X^w)$ is a Cartier divisor in $BS^Q$. Then there exists a surjective birational proper map $\psi_Q:BS^Q\rightarrow \widetilde{X^w}$ such that

\[
  \begin{tikzcd}
  BS^Q \arrow[swap]{rd}{\varphi_Q}\arrow{r}{\psi_Q} & \widetilde{X^w}\arrow{d}{\pi}\\
  & X^w
  \end{tikzcd}
  \]
\end{lemma}
\begin{proof}
By the universal property of blow-up maps, $\varphi_Q$ factors through the blow-up of $X^w$ along $\partial X^w$. The map $\psi_Q$ is birational since $\varphi_Q$ and $\pi$ are (see \ref{BS-birational}). Similarly, $\psi_Q$ is proper since $\varphi_Q$ is proper and $\pi$ is separated. Finally, since the universal property uses fibre diagrams, $\pi$ and $\varphi_Q$ surjective implies that $\psi_Q$ is surjective as well.
\end{proof}

\begin{corollary}\label{universal}
Suppose $\varphi_Q:BS^Q\rightarrow X^w$ is a generalized Bott-Samelson resolution where $BS^Q$ is factorial. Then there exists surjective proper birational map $\psi_Q:BS^Q\rightarrow \widetilde{X^w}$ which factors through the blow-up of $X^w$ along $\partial X^w$. In particular, choosing $Q=(w_1,...,w_k)$ such that each $w_i$ is a simple reflection provides one such map.
\end{corollary}
\begin{proof}
By dimension considerations, $\psi_Q^{-1}(\partial X^w)$ is a divisor in $BS^Q$. Since Cartier and Weil divisors are equivalent in factorial varieties, the result follows. Furthermore, choosing $Q$ so that the $w_i$ are simple reflections results in $BS^Q$ being smooth (by Lemma \ref{BS-smooth}) and hence factorial. 
\end{proof}

Understanding when $BS^Q$ is factorial requires an understanding of the same question for Schubert varieties. There is a pattern embedding description for this in \cite{Woo-Yong}, at least in the $GL_n$ case. One key takeaway from Corollary \ref{universal} is that there exist many such maps $\psi_Q$ when we choose $Q$ appropriately. The existence of these maps already provides us with new information regarding the $T$-fixed points of $\widetilde{X^w}$.

\begin{lemma}\label{equivariant}
Let $\pi:\widetilde{X^w}\rightarrow X^w$ be the blow-up of $X^w$ along $\partial X^w$. Then the $B$-action on $X^w$ extends to the total transform $\widetilde{X^w}$. Furthermore, $\psi_Q$ is a $B$-equivariant map.
\end{lemma}
\begin{proof}

Since $\partial X^w$ is $B$-invariant, the $B$-action extends to $\tilde{X}^w$ (see Proposition 3.9.1 of \cite{Kollar}). Furthermore, $\pi$ is $B$-equivariant. Since $\varphi_Q$ is also $B$-equivariant, then $\psi_Q$ is too.
\end{proof}

\begin{corollary}
$\widetilde{X^w}$ has finitely many $T$-fixed points.
\end{corollary}
\begin{proof}
By Lemma \ref{equivariant}, $T$-fixed points must map onto $T$-fixed points. The result follows from the fact that $BS^Q$ has finitely many $T$-fixed points.
\end{proof}

Our end goal is to show that understanding $\widetilde{X^w}$ restricts to an understanding of its $T$-fixed points. Since the $T$-fixed points of $BS^Q$ are easy to describe, the hope is that $\psi_Q$ can provide us with an explicit statement about the structure of $\widetilde{X^w}$.

\subsection{Extending the frobenius splitting}

Our proof that $\widetilde{X^w}$ is Gorenstein utilized the notion of Frobenius splittings and their relation to sections of the anticanonical bundle. It is natural to ask whether $\widetilde{X^w}$ is also Frobenius split.

There do exist results about extending Frobenius splittings to the total transform of a blow-up, such as in the work of Lakshmibai, Mehta, and Parameswaran in \cite{LMP}. This result however only works for the blow-up of a smooth Frobenius split variety along a smooth center (under certain mild vanishing conditions). In Proposition \ref{toric-frobenius}, we extended the Frobenius splitting in the toric case.

There are two natural approaches to extending the Frobenius splitting to the total transform. One perspective is to directly define a Frobenius splitting on the blow-up algebra

$$S[\epsilon_{1},...,\epsilon_{m+n}]/(I+\langle\epsilon_1,...,\epsilon_m\rangle+\Syz(\{g_i\}))$$
by defining a splitting map on monomials of the form $$C\epsilon_{1}^{a_1}\cdot...\cdot \epsilon_{m+n}^{a_{m+n}}$$

A second approach is to observe that $\widetilde{X^w}$ degenerates to a Frobenius split variety. That is, there is a flat morphism $F:\mathcal{Y}\rightarrow \mathbb{A}^1$ where $F^{-1}(0)=\mathcal{Y}_0$ is the toric scheme found in Section \ref{toric}, and $F^{-1}(t)=\mathcal{Y}_t$ is the general fibre which is isomorphic to $\widetilde{X^w}$. But is being Frobenius split open in flat families?

One might hope that there is a Frobenius splitting of the whole family $\mathcal{Y}$ which compatibly splits each $\mathcal{Y}_t$. However, there can only ever be finitely many  compatibly split subvarieties (see \cite{Kumar-Mehta}). On the other hand, it might be possible to extend the splitting of $\mathcal{Y}_0$ to a splitting of $\mathcal{Y}$ where $\mathcal{Y}_t$ is not necessarily compatibly split. One could then view the general fibre $\mathcal{Y}_t$ as a GIT quotient of $\mathcal{Y}$. Indeed, there is a $\mathbb{G}_m$ action on $\mathcal{Y}$, and if $\mathcal{Y}$ is Frobenius split, then so is $\mathcal{Y}_t \cong  \mathcal{Y}//\mathbb{G}_m$ (using a different splitting -- see Theorem 7.1 in \cite{Smith}). At present however, we don't know if $\widetilde{X^w}$ is Frobenius split. 

\subsection{Weakly normal varieties}\label{weakly_normal}

Recall that a variety $X$ is normal if every finite birational map $f:Y\rightarrow X$ is an isomorphism. We say that $X$ is \textbf{weakly normal} if every finite birational \textit{bijective} map is an isomorphism.
Bijective  here means that the reduction of the fibre over a point is again a point. In characteristic 0, the notions of weakly normal and seminormal are equivalent (see \cite{Vitulli} for more information).

A convenient way to show that a variety is weakly normal is to show that it is Frobenius split.

\begin{lemma}\label{frobenius-weaklynormal}
Frobenius split varieties are weakly normal.
\end{lemma}
\begin{proof}
See Proposition 1.2.5 in \cite{Brion-Kumar}.
\end{proof}

We would like to know that $\widetilde{X^w}$ is weakly normal in order to reduce questions involving $\psi_Q$ to information about the restricted map on $T$-fixed points. We rely on the fact that $\widetilde{X^w}$ degenerates to a weakly normal scheme, and that being weakly normal is an open condition in flat families.

\begin{proposition}\label{flat-weaklynormal}
Let $\mathcal{Y}$ be a flat family over $\mathbb{A}^1$ such that $\mathcal{Y}_0$, the special fibre over the origin, is weakly normal. Then $\mathcal{Y}_t$ is also weakly normal.
\end{proposition}
\begin{proof}
See \cite{Vitulli}.
\end{proof}

\begin{corollary}\label{BU-weaklynormal}
The blow-up of $X^w$ along $\partial X^w$ is weakly normal.
\end{corollary}
\begin{proof}
By Corollary \ref{commute}, the total transform $\widetilde{X^w\cap X_v^{\circ}}$ Gr\"{o}bner degenerates to a toric scheme, which by Lemma \ref{toric-frobenius} is Frobenius split. The result now holds by Lemma \ref{frobenius-weaklynormal} and Proposition \ref{flat-weaklynormal}, and using that being weakly normal can be checked locally.
\end{proof}

\subsection{Reducing to $T$-fixed points}

The next theorem will finally reduce understanding $\psi_Q$ to a question of $T$-fixed points. First we need a small lemma:

\begin{lemma}\label{torus_fixed}
Suppose that $X\hookrightarrow \mathbb{P}^n$ is an equivariant inclusion with respect to a torus action $T$. If $\dim X>0$ (and $X\neq \emptyset$), then $|X^T|>1$.
\end{lemma}
\begin{proof}
	See \cite{DaSilva-strict}.
\end{proof}

\begin{theorem}\label{T-fixed points}
Let $\psi_Q:BS^Q\rightarrow \widetilde{X^w}$ be as in Lemma \ref{universal}. Then $\psi_Q$ is an isomorphism iff it is a bijection on $T$-fixed points.
\end{theorem}
\begin{proof}

One direction is obvious. For the other direction, observe that $\psi_Q$ is birational by Lemma \ref{universal} and $\widetilde{X^w}$ is weakly normal by Corollary \ref{BU-weaklynormal}, so it is enough to show that $\psi_Q$ is finite and bijective to conclude that it is an isomorphism. By assumption, $\psi_Q$ is a bijection on $T$-fixed points, so by $B$-equivariance we can use the $B$-action to show that it is bijective at all points. Since $\psi_Q$ is also proper, we only need that $\psi_Q$ is quasi-finite for the result to hold.

Consider the fibre $Z_v = \psi_Q^{-1}(e_v)$, which is a projective subvariety of $BS^Q$. The embedding of $Z_v$ into projective space is equivariant with respect to the torus action on $Z_v$ (since $\psi_Q$ is $B$-equivariant). Under these conditions, dim($Z_v)>0 \Rightarrow |Z_v^T|>1$ by Lemma \ref{torus_fixed}. Since $|Z_v^T|=1$ by assumption, we conclude that $\dim(Z_v)=0$.

Since $\psi_Q$ is $B$-equivariant, the fibre of a point in $Be_vB/B$ is a $B$-translate of $Z_v$.  These Bruhat cells cover $X^w$, so $\psi_Q$ is quasi-finite.
\end{proof}

An understanding of when the Bott-Samelson resolution $\varphi_Q$ is an isomorphism over the smooth locus might be useful in determining information about $\psi_Q$. A discussion about these strict Bott-Samelson maps as well as examples on how reducing to $T$-fixed points can be a powerful tool can be found in  \cite{DaSilva-strict}.

While these results suggest that $\widetilde{X^w}$ is somewhere between a Bott-Samelson variety and a Schubert variety, the hope is that $\widetilde{X^w}$ is itself a Bott-Samelson variety or some simple quotient of one. This will be the topic of future research.

\bibliographystyle{plain}
\bibliography{Library}

\end{document}